\documentclass[12pt,reqno,twoside]{amsart}
\usepackage{amsfonts,amsmath,amssymb}
\usepackage{mathrsfs,mathtools}
\usepackage{enumerate}
\usepackage{hyperref}
\usepackage{esint}
\usepackage{graphicx}
\usepackage{bm}
\usepackage{commath}
\usepackage{esint}
\usepackage{placeins}
\usepackage{flafter}
\usepackage{algorithm}
\usepackage{algpseudocode}
\usepackage{tikz}
\usepackage[textsize=small]{todonotes}
\usepackage[dvips,bottom=1.4in,right=1in,top=1in, left=1in]{geometry}
\usepackage[latin1]{inputenc}
\usepackage{comment}

\DeclareMathAlphabet{\mathpzc}{OT1}{pzc}{m}{it}
\numberwithin{equation}{section}

\makeatletter
\def\eqnarray{\stepcounter{equation}\let\@currentlabel=\theequation
	\global\@eqnswtrue
	\tabskip\@centering\let\\=\@eqncr
	$$\halign to \displaywidth\bgroup\hfil\global\@eqcnt\z@
	$\displaystyle\tabskip\z@{##}$&\global\@eqcnt\@ne
	\hfil$\displaystyle{{}##{}}$\hfil
	&\global\@eqcnt\tw@ $\displaystyle{##}$\hfil
	\tabskip\@centering&\llap{##}\tabskip\z@\cr}
\def\endeqnarray{\@@eqncr\egroup
	\global\advance\c@equation\m@ne$$\global\@ignoretrue}

\newtheorem{theorem}{Theorem}[section]

\newtheorem{lemma}[theorem]{Lemma}

\numberwithin{equation}{section}

\title{ Bilevel Inverse Problems in Neuromorphic Imaging} 
\date{\today}

\thanks{
	This work is partially supported by NSF grants DMS-2110263, DMS-1913004, the Air Force Office of Scientific Research (AFOSR) under Award NOs: FA9550-22-1-0248 and FA9550-19-1-0036. 
}

\author{Harbir Antil}
\address{H. Antil. The Center for Mathematics and Artificial Intelligence
	and Department of Mathematical Sciences, George Mason University,
	Fairfax, VA 22030, USA.}
\email{hantil@gmu.edu}
\author{David Sayre}
\address{D. Sayre. The Center for Mathematics and Artificial Intelligence
	and Department of Mathematical Sciences, George Mason University,
	Fairfax, VA 22030, USA.}
\email{dsayre@gmu.edu}

\begin{document}				
	
	\begin{abstract}	
		Event or Neuromorphic cameras are novel biologically inspired sensors that record data 
		based on the change in light intensity at each pixel asynchronously. They have a temporal 
		resolution of microseconds.  This is useful for scenes with fast moving objects that can 
		cause motion blur in traditional cameras, which record the average light intensity over 
		an exposure time for each pixel synchronously. This paper presents a bilevel
		inverse problem framework for neuromorphic imaging. Existence of solution
		to the inverse problem is established. Second order sufficient conditions are derived
		under special situations for this nonconvex problem. A second order Newton type solver
		is derived to solve the problem. The efficacy of the approach is shown on several examples.
	\end{abstract}

	
	\maketitle

	\section{Introduction} \label{sec:Intro}
	
	Event (Neuromorphic) cameras are novel biologically inspired sensors that record data based on the change in light intensity at each pixel asynchronously \cite{Gallego_survey}. If the change in light intensity at a given pixel is greater than a preset threshold then an event is recorded at that pixel.  For this reason if there is no change to the scene, be that movement or brightening/dimming of a light source, no events will be recorded. In contrast if a scene is dynamic from camera movement or from the movement of an object in the scene then each pixel of the event camera will record intensity changes with a temporal resolution on the order of microseconds \cite{Gallego_survey}. This logging of events results in a non-redundant stream of events through the time dimension for each pixel \cite{DBLP:journals/corr/abs-1811-00386}. The stream of data is exceptionally useful for scenes with fast moving objects that can cause motion blur in traditional cameras, which record the average light intensity over an exposure time for each pixel synchronously \cite{pan2019bringing}. Figure \ref{f:nucamera} gives an illustration of the differences between traditional and event cameras. 
	\begin{figure}[htb]
		\centering
	\includegraphics[width=.5\textwidth]{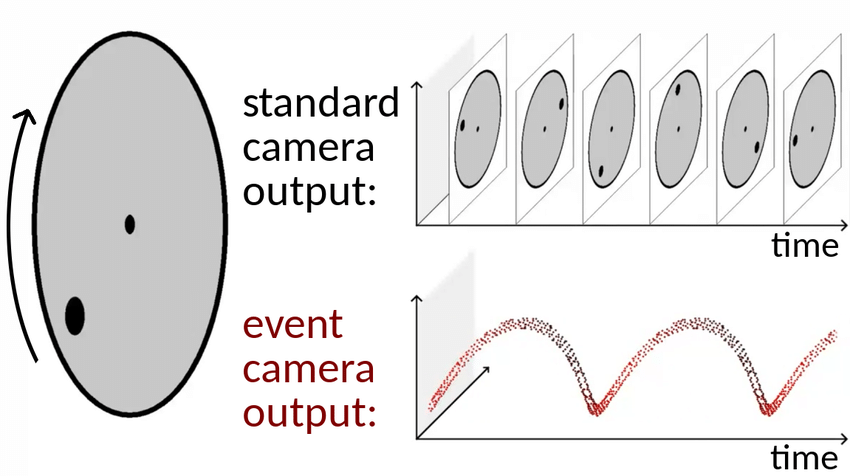}
		\caption{Comparison of the output of a standard frame-based camera and an 
			event camera when facing a black dot on a rotating disk. The standard camera 
			outputs frames at a fixed rate, thus sending redundant information when there 
			is no motion in the scene \cite{Gehrig_2018}.}
		\label{f:nucamera}
	\end{figure}
	
	The prior work, \cite{pan2019bringing}, proposed the Event-based Double Integral (EDI) and multiple Event-based Double integral (mEDI) algorithms to address motion blur for the underlying inverse problem. The EDI model utilizes a single blurry standard camera image, along with associated event data, to generate one or more clear images through the process of energy minimization. The mEDI model, an extension of the EDI model, utilizes multiple images in conjunction with event data to produce unblurred images. The EDI and mEDI models were an extension of the work from \cite{DBLP:journals/corr/abs-1811-00386} where a continuous-time model was initially introduced. While these algorithms represent a leap forward over earlier work, questions remain about the validity of the models in various scenarios (see below), regarding what situations solutions exist and what optimization procedures are appropriate in order to find such solutions. 
	In this paper we provide a mathematical foundation and further extensions to
	the previous methods by extending the model and providing a bilevel inverse 
	problem framework. Indeed, as shown below, the EDI and mEDI models 
	correspond to inverse problems in image deblurring \cite{arridge2019solving,bertero2021introduction}.
	In particular, 
	\begin{enumerate}
		\item We introduce a bilevel optimization problem that allows for the simultaneous 
		optimization over a desired number of frames. We also illustrate that this 
		optimization problem is indeed an inverse problem. Therefore, we use the terms
		`optimization' and `inverse', interchangeably.
		\item We establish existence of solution to this problem. Under certain conditions, we derive the 
		second order sufficient conditions and provide local uniqueness of solution to these nonconvex 
		problems.
		\item A fully implementable framework based on second order methods has been developed.
		The benefits of the proposed approach are illustrated using multiple numerical examples.
	\end{enumerate}
	For completeness, we emphasize that the bilevel approaches to search for the hyperparameters 
	in traditional imaging is not new, see for instance \cite{HAntil_ZDi_RKhatri_2020a,MR3592840}.
	However, the problem considered in this paper is naturally bilevel without having to search for the hyperparameters.
	
	\medskip
	\noindent 
	{\bf Outline:}
	The remainder of the paper is organized as follows. 
	Section~\ref{s:not} focuses on some notation and preliminary results.
	In Section~\ref{s:ebc}, we 
	discuss the basics of event based cameras. Section~\ref{s:existing} focuses on
	existing models from \cite{pan2019bringing} which serves as a foundation for
	the proposed bilevel optimization based variational model in Section~\ref{s:Opt} . Existence of 
	solution to the proposed optimization problem is shown in Theorem~\ref{thm:exist}. 
	Moreover, local convexity of our reduced functional is shown in Theorem~\ref{thm:interval}. 
	Section~\ref{s:numres} first focuses on how to prepare the data for the algorithm.
	It then discuss implementation details. The problem itself is solved using a second order
	Newton based method. Subsequently, several illustrative numerical examples are provided
	in Section~\ref{s:nex} .

	\section{Notation and Preliminaries}
	\label{s:not}
	
	Let $n_t$ represent the number of images and $n_x \times n_y$ denotes the 
	number of pixels per image. We use $\bm{U}\in \mathbb{R}^{n_t \times n_x \times n_y}$ 
	to denote a tensor. We use the vector 
	\[
	\bm{u}_{xy} \in \mathbb{R}^{n_t}, 
	\]
	to represent the $(x,y)$ pixel value of $\bm{U}$ for all times. 
	Moreover, we use the matrix 
	\[
	\bm{u}^{i} \in \mathbb{R}^{n_x \times n_y} . 
	\]
	to represent an image of size $\mathbb{R}^{n_x \times n_y}$  at a fixed time
	instance $i \in \{1,\dots, n_t\}$.
	We use the scalar
	\[
	\bm{u}^i_{xy} \in \mathbb{R} 
	\]
	to denote the value of $\bm{U}$ at a pixel location $(x,y)$ and time $i \in \{ 1, \dots, n_t \} $.
	Graphical representations are given in Figure \ref{f:ntn} . 
	Finally, given quantities $\bm{u}_{xy}$ and $\bm{M} \in\mathbb{R}^{{n_t}\times{\eta}}$ for some $\eta \in \mathbb{N}$, we define the operation $\langle \cdot , \cdot \rangle$ as:
	\[ \langle \bm{u}_{xy},\bm{M} \rangle := (\bm{u}_{xy})^\top \bm{M} \in \mathbb{R}^{\eta}.\] 
	 \begin{figure}[htb]
	 	\centering
	 	\includegraphics[width=0.6\textwidth]{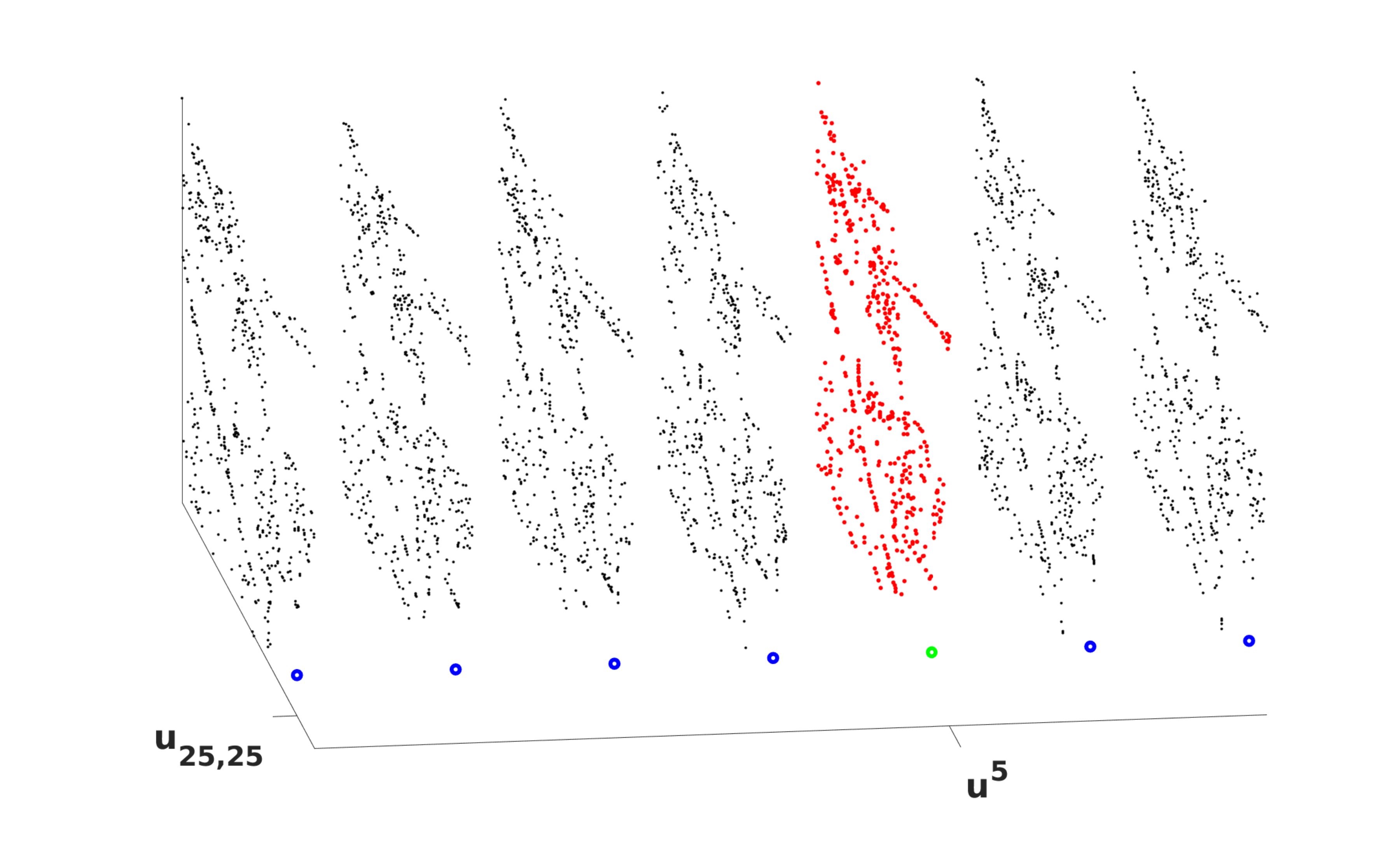}
	 	\caption{
			\label{f:ntn}
			The panel describes multiple images one of those images for instance is 
	 		$\bm{u}^i \in \mathbb{R}^{n_x \times n_y}$ and it is marked in red color. The 
	 		circles represent the vector $\bm u_{xy} \in \mathbb{R}^{n_t}$ 
	 		and are marked in blue color. Green circle indicates an overlap between 
	 		$\bm{u}^i$ and $\bm u_{xy}$.
	 	}
	 \end{figure}
	
	The remainder of the section is organized as follows. First in Section~\ref{s:ebc}, 
	we discuss the basic working principle behind the neuromorphic cameras. 
	The main ideas behind the algorithms presented in \cite{pan2019bringing} are 
	described in Section~\ref{s:existing}. We briefly discuss some potential limitations 
	of this approach which motivates our algorithm in Section~\ref{s:Opt}.

	\subsection{Neuromorphic (Event Based) Cameras}\label{s:ebc}
	
	\subsubsection{Neuromorphic Camera Basics}\label{sec:Basics}

	Neuromorphic cameras are composed of independent pixels that detect light intensity changes 
	in the environment as they occur. Light intensity is sampled $\mathcal{O}(\mu s)$ and  events 
	are logged if the intensity of the light is beyond a preset hardware defined threshold, $c$.
	We denote the instantaneous light intensity at pixel location $(x,y)$ at time $s$ by 
		$\bm{q}^{s}_{xy} \in \mathbb{R}$. 
	When the light intensity detected by the camera exceeds the threshold for a given pixel located at $(x,y)$, 
	at a given time $s$, an event is logged. 
	Then the reference intensity, $\bm{q}^{s_{ref}}_{xy}$, for the pixel located at $(x,y)$ is updated. Due to the high rate of sampling, 
	and independent nature of the camera's pixels neuromorphic cameras are less susceptible to 
	exposure issues, as well as image blurring. Figure~\ref{f:nucamera} illustrates the difference in the data recorded between traditional and event based cameras.

	\subsubsection{Data representation}
	The output of an event based camera over some time interval, $\mathcal{I}$, is of the form $(s,x,y,p)$, with the following definitions: 
	\begin{enumerate}[$\bullet$]
		\item $s \in \mathcal{I} = \{s_1,s_2, \dots, s_n \, : \, s_1 < s_2 < \dots < s_n \}$,  	 	
		where $\mathcal{I}$ represents the exposure interval, each $s_j$ represents a discrete time at which an 
		event occurred and $n$ represents the total number of events recorded during the interval, $\mathcal{I}$. 
		When discussing methods that involve multiple images, we will denote the exposure interval for the 
		$i$-$th$ image as $\mathcal{I}_i$. We will also denote the set of events associated to each image $i$ as 
			$\{ s^i_j\}_{j=1}^{n_i}$. With this notation we have $i = 1, \dots, n_t$ and $j = 1, \dots, n_i$ with $n_i$ representing 
		the total number of events associated to image $i$. 
		
		\item $x,y$ represent pixel coordinates.  
		\item $p_{x,y}^{s_j} := \begin{cases}
		+1, & \log \left( \frac{\bm{q}^{s_j}_{xy}}{\bm{q}^{s_{ref}}_{xy}}\right) \geq c\\
		-1, & \log \left( \frac{\bm{q}^{s_j}_{xy}}{\bm{q}^{s_{ref}}_{xy}}\right) \leq -c \, 
		\end{cases}$ is the polarity of the event. 
		An increase in light intensity above a threshold, $c>0$, we regard as a polarity 1 
		and a decrease of intensity as a polarity shift of $-1$. No event is logged if 
		$p_{x,y}^{s_j} \in (-c,c)$.
		
	\end{enumerate}
	We can reformulate these events into a datacube with the following definition: 
	\begin{equation}\label{eq:cube}	
	{\bm{EC} }_{xy}^{s_j} = p_{x,y}^{s_j}, \quad j = 1, \dots, n ,
	\end{equation}
	which results in a sparse 3-dimensional matrix with entries of $0,1,-1$.
	We refer to Figure \ref{f:ex1} for 2D and 3D examples of the datacube.
	Representing the data in this way allows us to track intensity changes over 
	time on a per pixel basis.
	
	\subsection{Inverse problems in Image Deblurring} \label{s:inv}
	Image deblurring is a classic example of an inverse problem, where the goal is to reconstruct the unknown original image from a degraded, blurred observation. In image deblurring, the degradation process can be modeled as a interaction between the original image and a blur kernel which represents the blur caused by various factors such as movement \cite{DBLP:journals/corr/abs-2201-10522,zhang2022new}.
	
	In the context of image deblurring, the field can be divided into two categories, blind and non-blind, based on the amount of information known about the blur kernel \cite{almeida2009blind,satish2020comprehensive,tang2014non,xie2019non,DBLP:journals/corr/abs-2201-10522}. In the non-blind case, the blur kernel is known which leads to a simpler ableit ill-posed inverse problem \cite{tang2014non,xie2019non}. In the blind case, no information about the blur kernel is available, making the deblurring problem more challenging as the blur kernel must be estimated from the degraded image in addition to the original image \cite{satish2020comprehensive,DBLP:journals/corr/abs-2201-10522}. The semi-blind case is a particular instance of a blind image deblurring problem in that some but not all information about the blur kernel is available \cite{buccini2018semiblind,morin2013semi,park2012semi}. The problems under consideration in this project are of semi-blind type.
	
	\subsection{Existing Model, Algorithm, and Limitations}\label{s:existing}
	The approach discussed in this paper is motivated by \cite{pan2019bringing}. The
	article \cite{pan2019bringing} seeks to find latent un-blurred images using 
	the Event Based Double Integral (EDI) and multi-Event Based Double Integral (mEDI) 
	models. We briefly discuss this next. This will be followed by our new proposed models.
	
	\subsubsection{EDI Model}
	
	In \cite{pan2019bringing}, the discrete events outlined in Section \ref{sec:Basics} are used to  
	generate a continuous function, $\bm{e}_{xy}$, for each $(x,y)$ pixel location. This is done by 
	generating a series of continuous unit bump functions centered at each $s_j$: 
	$\phi_{s_j}(t)$. Then we can define $\bm{e}_{xy} : \mathbb{R} \rightarrow \mathbb{R}$ as:
	\begin{equation}\label{eq:exy}
	\bm{e}_{xy}(t) := \sum_{s_j \in \; \mathcal{I}} {\bm{EC}}_{xy}^{s_j} \cdot  \phi_{s_j}(t) \in \mathbb{R} .
	\end{equation}
	We graphically illustrate building this function for a single pixel below with the following string of events:
	\[
	[1,0,1, -1, 1].
	\] 
	In Figure \ref{f:stemplot} we see the function $\bm{e}_{xy}$ overlayed onto a stem plot of events.
	\begin{figure}[!htb]
		\centering
		\includegraphics[width=0.6\textwidth]{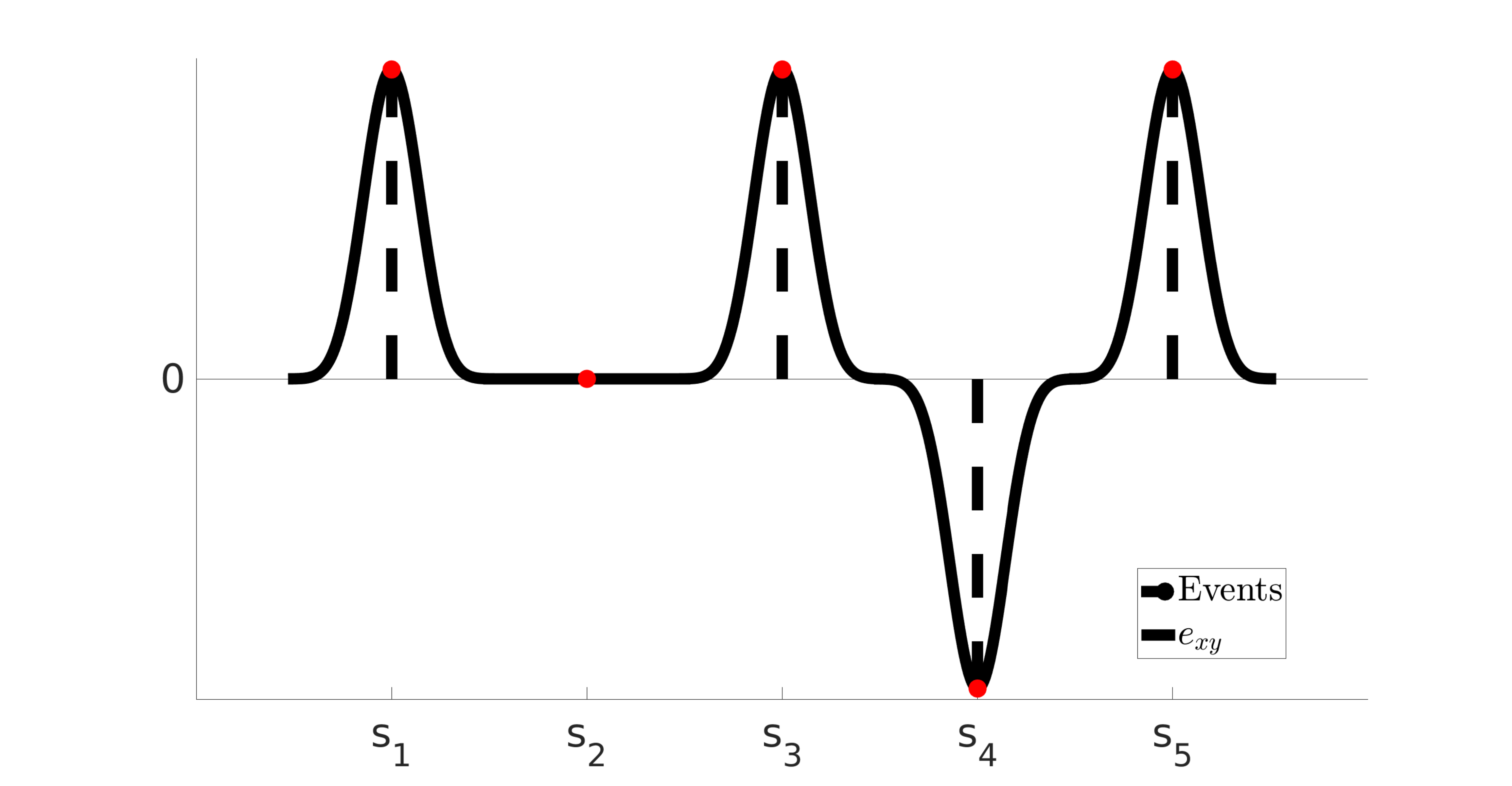}
		\caption{In this figure the solid line represents the function $\bm{e}_{xy}$ overlayed 
			onto the a stem plot (dashed line)that represents a series of recorded events. The red dots are used to highlight the events.}
		\label{f:stemplot}
	\end{figure}
	
	From \ref{eq:exy}, we define the `\emph{sum of events}' from time of interest $t_i$ to 
	arbitrary time $t$ as:
	\begin{equation}\label{eq:Exy}
	\bm{E}^i_{xy}(t,t_i) = \int_{t_i}^{t}\bm{e}_{xy}(r) \; dr , 
	\end{equation}
	where $\bm{E}^i_{xy}(t,t_i) \in \mathbb{R}$ corresponds to one pixel and  
		$\bm{E}^i(t,t_i) = \{ \bm{E}^i_{xy}(t,t_i) \}_{xy}  \in \mathbb{R}^{{n_x} \times {n_y}}$ 
		consists of all the pixels. 
	We note that $t_i$, $i=1, \dots, n_t$ are manually chosen and represent the time at which we wish 
	to generate an image. In practice we will set each $t_i$ to correspond to the timestamps of 
	the standard camera images we are using in the model. For convenience we will omit $t_i$ 
	from $\bm{E}^i_{xy}(t,t_i)$ unless needed for clarity.
	Introducing a scalar variable $z$ 
	gives us the EDI Model: 
	\begin{equation}\label{eq:B}
	\bm{B}^i_{xy} = \frac{1}{|\mathcal{I}_i|} \int_{t_i-|\mathcal{I}_i|/2}^{t_i+|\mathcal{I}_i|/2} \bm{L}^i_{xy} \cdot \exp \big(z \cdot \bm{E}^i_{xy}(t)\big)dt \, .
	\end{equation}
	Where $\bm B^i_{xy} \in \mathbb{R}$ and $\bm B^i =\{ \bm B^i_{xy}\}_{xy} \in \mathbb{R}^{n_x\times n_y}$ is a standard camera image. Furthermore, $\bm L^i_{xy} \in \mathbb{R}$ 
	is the pixel value of the unblurred image at location $(x,y)$, and $|\mathcal{I}_i|$ is the exposure time for the $i^{th}$ image. We note that 
	the superscript $i$ in $\bm{B}^i_{xy}$ and $\bm{L}^i_{xy}$ are simply placeholders for notation since the EDI model 
	only considers a single image at a time.
	Taking the log and re-arranging terms of \eqref{eq:B} leads to:
	\begin{equation}\label{eq:EDI}
	\bm{v}_{xy}^i = \bm{d}_{xy}^i - \bm{g}_{xy}^i(z) , 
	\end{equation}
	where 
	\begin{align}\label{eq:vxy}
	\begin{aligned}
	\bm{v}_{xy}^i &=  \log(\bm{L}_{xy}^i), \quad
	\bm{d}_{xy}^i = \log \bm{B}_{xy}^i,\\
	\bm g^i_{xy}(z) &= \log\bigg(\frac{1}{|\mathcal{I}_i|} \int_{t_i-|\mathcal{I}_i|/2}^{t_i+|\mathcal{I}_i|/2} \exp \big(z \cdot \bm{E}_{xy}^i(t)\big)dt\bigg).
	\end{aligned}
	\end{align} 
The formulation in \eqref{eq:EDI} gives us the correlation between the EDI model and other semi-blind inverse problems in image deblurring. Here $\bm{B^i}$ is our blurry image, $\bm{L^i}$ is our latent unblurred image, and $\bm{g^i}$ is our kernel function. As shown in \cite{almeida2009blind,arridge2019solving, bertero2021introduction,bungert2020robust,chung2011designing} this formulation is common for many inverse problems. We note the EDI/mEDI models would fall into the category of semi-blind inverse problems mentioned in Section~\ref{s:inv} since the kernel, $\bm{g^i}$, is known up to the parameter $z$. 	
	\subsubsection{mEDI Model}
	If we consider two latent images $\bm{v}^i$ and $\bm{v}^{i+1}$ we know that event data 
	will give us the per pixel changes that have occurred between the two images. We can 
	mathematically describe this change as $\int_{t_i}^{t_{i+1}} \bm{e}_{xy}(s) ds$ where $\bm{e}_{xy}$ 
	is as given in \eqref{eq:exy} and $t_i$ are user chosen times of interest. Incorporating 
	the threshold variable $z$ allows us to formulate the following equations 
	\begin{equation}\label{eq:vi}
	\bm{v}^{i+1}_{xy} - \bm{v}^i_{xy} = z \int_{t_i}^{t_{i+1}} \bm{e}_{xy}(s) ds =: \bm{b}_{xy}^i(z).
	\end{equation}
	Combining \eqref{eq:EDI} and \eqref{eq:vi} gives us the following system 
	of equations:
	\begin{equation}\label{eq:system}
	\left(
	\begin{array}{rrrrrrrr}
	-1  &   1 &      &     &  & &  &  \\
	&  -1 &  1  &     &  &  &  &  \\
	&     &  & \ddots & \ddots &   &  &   \\
	&     &     &    &   & -1 & 1 &   \\		
	&     &     &    &   &    & -1 & 1 \\
	\hline
	1 &  &  &  &  &  &  & \\
	& 1  &  &  &  &  &  & \\ 
	& &  1  &  &  &  &  & \\
	& &  &\ddots & & & & \\
	& &  &  &  1  &  &  & \\
	& &  &  &  &  1  &  & \\
	& &  &  &  &  &  1  & 
	\end{array}	
	\right){\bm{v}}_{xy} = \left(\begin{array}{rrr}
	\bm{b}_{xy}\\
	\bm{d}_{xy} - \bm{g}_{xy}
	\end{array}
	\right)
	\end{equation} 
	Notice that, \eqref{eq:system} gives us the mEDI model described by \cite{pan2019bringing}.
	The size the upper half of the submatrix in \eqref{eq:system} is $\mathbb{R}^{(n_t-1) \times n_t}$ and the lower half matrix is $\mathbb{R}^{{n_t} \times {n_t}}$.
	
	\subsubsection{mEDI Model - Optimizing for z}
	The variable $z$ is found by solving the following minimization problem for each 
	pixel $(x,y)$ of an image of interest $i$: 
	\begin{equation}\label{eq:medi_opt}
	\min_z \frac12 \|{\bm{v}}^i_{xy}({z}) + {\bm{g}}^i_{xy}(z) - {\bm{d}}^i_{xy}\|^2_2.
	\end{equation}
	Here $\bm{v}_{xy}$, ${\bm{d}}^i_{xy}$, and ${\bm{g}}^i_{xy}(z)$ are as given in 
	\eqref{eq:vxy}. We refer to \cite{pan2019bringing} for their solution approach. 
	Notice that \eqref{eq:vi} is not used as constraint to the minimization problem 
	\eqref{eq:medi_opt}, but \eqref{eq:medi_opt} is used to generate the image
	reconstructions.

	\subsubsection{Limitations of the mEDI and EDI Models} 
	While the mEDI model uses multiple frames in order to accurately represent 
	de-blurred images some limitations for the model remain.
	
	\begin{enumerate}[(a)]
		\item Performance of the model can degrade for images that contain large 
		variations of light intensity within a single scene. Some examples of this 
		are having very dark objects in a bright room with a highly dynamic scene. 
		In these instances we can see ``shadowing" effects where shadows appear 
		across multiple frames. We refer to Section \ref{s:shadow} for specific examples. 
		
		\item 
		The optimization problem introduced \eqref{eq:medi_opt} to 
		identify $z$ only seeks to optimize the pixels of one image at a time and neglects 
		the constraints \eqref{eq:vi}. Then uses \eqref{eq:vi} to reconstruct multiple images
		with the same $z$ value.
		
		\item Recall from \ref{eq:vi} that $\bm{b}^i_{xy} := z \int_{t_i}^{t_{i+1}} \bm{e}_{xy}(s) \; ds$. This definition does not include any information from the standard images $\bm{B}^i$. Due to this the mEDI model is not well suited for generating image reconstructions across time-spans that include multiple standard images. 
	\end{enumerate}
	\section{Proposed Model and Algorithm}\label{s:Opt}
	
	\subsection{Model}
	
	For each pixel $(x,y)$, we consider the following bilevel inverse problem 
	\begin{subequations}\label{eq:orig}
		\begin{equation}\label{eq:out}
		\min_{\bm z_{xy}} J(\bm u_{xy},\bm{z}_{xy}) := \frac{1}{2} \| \bm u_{xy}(\bm{z}_{xy}) + \bm g_{xy}(\bm{z}_{xy}) - \bm d_{xy} \|_2^2 + \frac{\lambda_1}{2}\|\bm{z}_{xy} \|_2^2
		\end{equation}
		subject to constraints 
		\begin{equation}\label{eq:in}
		\min_{\bm u_{xy}} \frac{1}{2} \| {\bm A} \bm u_{xy}  -  \bm b_{xy}(\bm{z}_{xy}) \|_2^2 + \frac{\lambda_2}{2} \| \bm u_{xy}\|_2^2 .
		\end{equation}
	\end{subequations}
	Notice, that this is an optimization problem for one pixel across all images. 
	By $0 < \lambda_1,\lambda_2$, we denote the regularization parameters. Moreover, 
	\[
	\bm A = 
	\left(
	\begin{array}{rrrrrrrr}
	-1  &   1 &      &     &  & &  &  \\
	&  -1 &  1  &     &  &  &  &  \\
	&     &  & \ddots & \ddots &   &  &   \\
	&     &     &    &   & -1 & 1 &   \\		
	&     &     &    &   &    & -1 & 1 \\
	\end{array}	
	\right) \in \mathbb{R}^{(n_t-1) \times n_t} 
	\]	
	is a submatrix of the matrix given in \eqref{eq:system}. 
	Recall in Equation \eqref{eq:medi_opt} that $z$ is taken as a scalar and the optimization is done over a single image at a time. We consider  $\bm{z}_{xy}$ as size $n_t \times 1$ so that we can optimize 
	over the same pixel $(x,y)$ for all images simultaneously.  Also, in contrast to \cite{pan2019bringing} 
	since we have defined $\bm{z}_{xy}$ as a vector over all images, we also modify the variable 
	$\bm{b}$ with the following definition:
	\begin{equation}\label{eq:b}
	\bm{b}^i_{xy}(\bm{z}_{xy}) := (\bm{d}_{xy}^{i+1} - \bm{g}_{xy}^{i+1}(\bm{z}_{xy})) - (\bm{d}_{xy}^{i} - \bm{g}_{xy}^{i}(\bm{z}_{xy})) \in \mathbb{R}^{n_t -1}.
	\end{equation}
	Where $\bm{z}^i_{xy} \in \mathbb{R}$ and is the $i$-$th$ component of the vector $\bm{z}_{xy} \in \mathbb{R}^{n_t}$. This is done so the events associated to the $i$-$th$ image are being optimized by the $i$-$th$ $\bm{z}_{xy}$ component.
	
	Notice that, $\bm A$ above is not a square matrix. For every $\lambda_2 > 0$, 
	the lower level problem is uniquely solvable and the solution is given by 
	\begin{equation}
	\bm{u}_{xy}(\bm{z}_{xy}) 
	= (\bm A^\top \bm A  + \lambda_2 \bm I)^{-1} \bm A^\top \bm b_{xy}(\bm{z}_{xy}) 
	= \bm K^{-1} \bm{A^\top}\bm b_{xy}(\bm{z}_{xy})\label{eq:u}
	\end{equation}
	with
	\[ 
	\bm K :=  (\bm A^\top \bm A  + \lambda_2 \bm I) \in \mathbb{R}^{n_t \times n_t}.
	\] 
	Substituting, this in 
	the upper level problem, we obtain the following reduced problem 
	\begin{equation}\label{eq:rp}
	\min_{\bm{z}_{xy} \in Z_{ad}} \mathcal{J}(\bm{z}_{xy}) := J(\bm K^{-1} \bm A^\top \bm b_{xy}(\bm{z}_{xy}),\bm{z}_{xy}) 
	\end{equation}
	which is now just a problem in $\bm{z}_{xy}$. Notice, that the reduced problem \eqref{eq:rp}
	is equivalent to the full space formulation
	\begin{subequations}\label{eq:fs_1}
		\begin{equation}\label{eq:out_1}
		\min_{\bm{z}_{xy}} J(\bm u_{xy}, \bm{z}_{xy}) 
		\end{equation}
		subject to constraints 
		\begin{equation}\label{eq:in_1}
		\bm K \bm u_{xy} = \bm A^\top \bm b_{xy} (\bm{z}_{xy}) .
		\end{equation}
	\end{subequations}
	The next result establishes existence of solution to the reduced problem 
	\eqref{eq:rp} and equivalently to \eqref{eq:fs_1}.
	\begin{theorem}\label{thm:exist}
		There exists a solution to the  reduced problem \eqref{eq:rp}.
	\end{theorem}
	
	\begin{proof}
		The proof follows from the standard Direct method of calculus of variations. We will 
		omit the subscript $xy$ for notation simplicity. 
		We begin by noticing that $\mathcal{J}(\cdot)$ is bounded below by 0. Therefore, there exists 
		a minimizing sequence $\{\bm z_n\}_{n \in \mathbb{N}}$ such that 
		$$
		\lim_{n\rightarrow \infty} \mathcal{J}(\bm z_n) = \inf_{\bm z} \mathcal{J}(\bm z) .
		$$
		It then follows that 
		\[
		\frac{\lambda_1}{2} \lim_{n\rightarrow \infty} \| \bm z_n\|^2 
		\le \lim_{n\rightarrow \infty} \mathcal{J}(\bm z_n) 
		= \inf_{\bm z} \mathcal{J}(\bm z) 
		\le \mathcal{J}(\bm 0) 
		= C < \infty ,
		\]
		where the constant $C$ is independent of $n \in \mathbb{N}$. Thus, $\{\bm z_n\}_{n \in \mathbb{N}}$
		is a bounded sequence. Therefore, it has a convergent subsequence, still denoted by 
		$\{\bm z_n\}_{n \in \mathbb{N}}$, such that $\bm z_n \rightarrow \bar{\bm z}$. 
		It them remains to show that $\bar{\bm z}$ is the minimizer. Since, 
		$\bm u_n := {\bm u}(\bm z_n)$ solves \eqref{eq:u}, therefore, we have that 
		$\bm u_n \rightarrow {\bm u}(\bar{\bm z})$ and $\bm g(\bm z_n) \rightarrow 
		\bm g(\bm {\bar z})$ as $n \rightarrow \infty$. Subsequently, from the definition of 
		infimum and these convergence estimates, we readily obtain that 
		\[
		\lim_{n\rightarrow \infty} \mathcal{J}(\bm z) 
		\ge \liminf_{n\rightarrow \infty} \frac12 \| \bm u(\bm z_n) + \bm{g}(\bm z_n) - \bm d \|_2^2
		+ \liminf_{n\rightarrow \infty} \frac{\lambda_1}{2} \| \bm z_n \|_2^2 
		= \mathcal{J}(\bar {\bm z}) . 
		\]
		Thus $\bar {\bm z}$ is the minimizer and the proof is complete. 	
	\end{proof}
	In order to develop a gradient based solver for  \eqref{eq:rp}, we next write 
	down the expression of gradient of $\mathcal{J}$. 
	\begin{lemma}\label{lem:grads}
		The reduced objective $\mathcal{J}$ is continuously differentiable 
		and the derivative is given by
		\begin{align*}
		\nabla \mathcal{J}(\bm{z}) 
		= \left\langle \bm K^{-1} \bm A^\top \bm b(\bm{z}_{xy}) + \bm g(\bm{z}_{xy}) - \bm d_{xy} \,  , \,
		\bm K^{-1} \bm A^\top \bm b_{xy}'(\bm{z}_{xy}) +  \bm g_{xy}'(\bm{z_{xy}}) \right\rangle
		+ \lambda \bm{z}_{xy},
		\end{align*}
		where $\bm K^{-1} \bm A^\top \in \mathbb{R}^{n_t \times (n_t-1)}$, 
		\begin{align*}
		\bm b_{xy}'(\bm{z}_{xy}) =  \begin{pmatrix}
		(\bm{g}_{xy}^1)' & - (\bm{g}_{xy}^2)' & 0 & \cdots & 0 & 0 & 0 \\
		0 & (\bm{g}_{xy}^2)' & - (\bm{g}_{xy}^3)' & 0 & \cdots & 0 & 0 \\
		0 & 0 & (\bm{g}_{xy}^3)' & - (\bm{g}_{xy}^4)' & 0 & \cdots &0 \\
		\vdots & \vdots & \vdots & \vdots & \vdots & \vdots & \vdots\\
		0 & 0 & 0 &  0 & 0 & (\bm{g}_{xy}^{n_t -1})' & - (\bm{g}_{xy}^{n_t})'
		\end{pmatrix} \in \mathbb{R}^{(n_t-1) \times n_t},
		\end{align*} and
		$\bm g_{xy}'(\bm{z_{xy}})$ is a diagonal matrix given by 
		\begin{align}
		&\bm g_{xy}'(\bm{z_{xy}}) = \\
		&\begin{pmatrix}
		\frac{\int_{ \mathcal{I}_1 } \bm{E}^1_{xy}(t) \exp(\bm{z}^1_{xy} \cdot \bm{E}^1_{xy}(t)) dt}{\int_{ \mathcal{I}_1 }\exp(\bm{z}^1_{xy} \cdot \bm{E}^1_{xy}(t)) dt} & 0 & 0 & 0 & 0\\
		0 & \frac{\int_{ \mathcal{I}_2} \bm{E}_{xy}^2(t) \exp(\bm{z}^2_{xy} \cdot \bm{E}_{xy}^2(t)) dt}{\int_{ \mathcal{I}_2}\exp(\bm{z}^2_{xy} \cdot \bm{E}_{xy}^2(t)) dt} & 0 & 0 & 0 \\
		0 & 0 & 0 &\ddots & 0 \\ 
		0 & 0 & 0 & 0 &  \frac{\int_{ \mathcal{I}_{n_t}} \bm{E}_{xy}^{n_t}(t) \exp(\bm{z}^{n_t}_{xy} \cdot \bm{E}_{xy}^{n_t}(t)) dt}{\int_{ \mathcal{I}_{n_t}}\exp(\bm{z}^{n_t}_{xy} \cdot \bm{E}_{xy}^{n_t}(t)) dt}
		\end{pmatrix}. \label{eq:c_exp}
		\end{align}
		
	\end{lemma}
	\begin{proof}
		The proof follows by simple calculations.	
	\end{proof}
	
	The next result establishes that $(\bm{g}^i_{xy}(\bm{z}_{xy}))''$ is bounded and strictly positive
	for each $i$, with $\mathcal{I}_i$ denoting the $i$-$th$ exposure time interval. 
	\begin{theorem}\label{thm:gconvex}
		Let $\mathcal{I}_i$ denote the $i$-$th$ exposure time interval with  
			$i = 1,\dots, n_t$. If there exists $t \in  \mathcal{I}_i$, for all $i$, with 
			${\bm E}^i_{xy}(t)$ being nonzero then $(\bm{g}^i_{xy}(\bm{z}_{xy}))'' \in \mathbb{R}$ 
			is bounded and strictly positive for all $i$.
	\end{theorem}
	\begin{proof}
		First we will calculate $(\bm{g}^i_{xy}(\bm{z}_{xy}))''$.
		Recall:
		\begin{align*}
		\bm{g}^i_{xy}(\bm{z}_{xy}) = \log\bigg(\frac{1}{| \mathcal{I}_i |}\int_{ \mathcal{I}_i }\exp(\bm{z}^i_{xy} \cdot \bm{E}^i_{xy}(t))dt\bigg)
		\implies 
		(\bm{g}^i_{xy}(\bm{z}_{xy}))' =
		\frac{\int_{ \mathcal{I}_i } \bm{E}^i_{xy}(t) \exp(\bm{z}^i_{xy} \cdot \bm{E}^i_{xy}(t)) dt}{\int_{ \mathcal{I}_i }\exp(\bm{z}^i_{xy} \cdot \bm{E}^i_{xy}(t)) dt} ,
		\end{align*}
		yielding 
		\begin{align}
		\begin{aligned}
		(\bm{g}^i_{xy}(\bm{z}_{xy}))''
		=\frac{ \boxed{\textrm{I}} - 
			\bigg(
			\int_{ \mathcal{I}_i } \bm{E}^i_{xy}(t) \exp(\bm{z}^i_{xy} \cdot \bm{E}^i_{xy}(t)) dt\bigg)^2 
		}
		{\big[\int_{ \mathcal{I}_i } \exp(\bm{z}^i_{xy}\cdot \bm{E}^i_{xy}(t))dt\big]^2} ,
		\label{eq:int}
		\end{aligned}
		\end{align}
		where $\boxed{\textrm{I}} = \bigg(\int_{ \mathcal{I}_i } \exp(\bm{z}^i_{xy} \cdot \bm{E}^i_{xy}(t)) dt\bigg)\bigg(\int_{ \mathcal{I}_i } \big[\bm{E}^i_{xy}(t)\big]^2 \exp(\bm{z}^i_{xy} \cdot \bm{E}^i_{xy}(t)) dt\bigg)$. 
		\linebreak
		 Let $\bigg(\int_{ \mathcal{I}_i } \exp(\bm{z}^i_{xy} \cdot  \bm{E}^i_{xy}(t)) dt\bigg) =: \alpha$. 
		Then we have the following:
		\begin{align*}
		\boxed{\textrm{I}}
		&= \bigg(\int_{ \mathcal{I}_i } \big[\bm{E}^i_{xy}(t)\big]^2 \exp(\bm{z}^i_{xy} \cdot \bm{E}^i_{xy}(t))\cdot\alpha\; dt\bigg)\\
		&= \bigg(\int_{ \mathcal{I}_i } \big[\bm{E}^i_{xy}(t)\big]^2 \exp(\bm{z}^i_{xy} \cdot \bm{E}^i_{xy}(t))\cdot \bigg(\int_{ \mathcal{I}_i } \exp(\bm{z}^i_{xy} \cdot \bm{E}^i_{xy}(s)) ds\bigg) \; dt\bigg)\\
		&= \bigg(\int_{ \mathcal{I}_i }\int_{ \mathcal{I}_i } \big[\bm{E}^i_{xy}(t)\big]^2 \exp(\bm{z}^i_{xy} \cdot \bm{E}^i_{xy}(t))\cdot  \exp(\bm{z}^i_{xy} \cdot \bm{E}^i_{xy}(s)) ds \; dt\bigg) . 
		\end{align*}
		Using Fubini's Theorem and following the procedure above for the two integral products in the numerator of \eqref{eq:int} we get the following equalities: 
		\begin{align*}
		\eqref{eq:int}
		&=\frac{
			\boxed{\textrm{I}} 
			- \bigg(\iint\limits_{ \mathcal{I}_i  \times  \mathcal{I}_i } \bm{E}^i_{xy}(t) \exp(\bm{z}^i_{xy} \cdot \bm{E}^i_{xy}(t)) \cdot\; \bm{E}^i_{xy}(s)  \exp(\bm{z}^i_{xy} \cdot \bm{E}^i_{xy}(s))\;ds\; dt\bigg)}{\big[\int_{ \mathcal{I}_i } \exp(\bm{z}^i_{xy}\cdot \bm{E}^i_{xy}(t))dt\big]^2}\\
		&=\frac{\bigg(\iint\limits_{ \mathcal{I}_i  \times  \mathcal{I}_i } \big[\exp(\bm{z}^i_{xy} \cdot \bm{E}^i_{xy}(t)) \exp(\bm{z}^i_{xy} \cdot \bm{E}^i_{xy}(s))\big]\bigg(\big[\bm{E}^i_{xy}(t)\big]^2 - \bm{E}^i_{xy}(s)\bm{E}^i_{xy}(t)\bigg) \;ds \; dt\bigg)}{\big[\int_{ \mathcal{I}_i } \exp(\bm{z}^i_{xy} \cdot \bm{E}^i_{xy}(t))dt\big]^2}\\
		&= \frac{\bigg(\iint\limits_{ \mathcal{I}_i  \times  \mathcal{I}_i } \big[\exp(\bm{z}^i_{xy} \cdot \bm{E}^i_{xy}(t)) \exp(\bm{z}^i_{xy} \cdot \bm{E}^i_{xy}(s))\big]\bigg(\frac{1}{2}\big[\bm{E}^i_{xy}(t) - \bm{E}^i_{xy}(s)\big]^2\bigg) \;ds \; dt\bigg)}{\big[\int_{ \mathcal{I}_1 } \exp(\bm{z}^i_{xy}\cdot \bm{E}^i_{xy}(t))dt\big]^2}.\stepcounter{equation}\tag{\theequation}\label{eq:pos}
		\end{align*}
		We notice that  Equation \eqref{eq:pos} is positive for $\bm{E}^i_{xy}(t) \neq \bm{E}^i_{xy}(s)$ (since all terms are positive) and the expression is $0$ for $t =s$. By assumption there exists some $t \in \mathcal{I}_i$ such that $\bm{E}^i_{xy}(t) \neq 0$ which implies that an event has occurred at the pixel ${xy}$ during the interval $\mathcal{I}_i.$ Without loss of generality suppose a single event occurred at time $\bar{t}$. Recall from \eqref{eq:Exy} that $\bm{E}^i_{xy}(t,t_i) = \int_{t_i}^{t}\bm{e}_{xy}(r) \; dr$. Then $0 =  \bm{E}^i_{xy}(t_i,t_i) \neq \bm{E}^i_{xy}(\bar{t},t_i)$ which implies  $\bm{E}^i_{xy}(t_i) \neq \bm{E}^i_{xy}(\bar{t})$. Which further impiles that $\eqref{eq:pos} > 0$. 
		Thus we conclude $(\bm{g}^i_{xy}(\bm{z}_{xy}))'' > 0$.
		Next, using \eqref{eq:pos}, we obtain the following upper bound:
		\begin{align*}
		(\bm{g}^i_{xy}(\bm{z}_{xy}))'' 
		&\leq \frac{\bigg(\iint\limits_{ \mathcal{I}_i \times  \mathcal{I}_i} \big[\exp(\bm{z}^i_{xy} \cdot \bm{E}^i_{xy}(t)) \exp(\bm{z}^i_{xy} \cdot \bm{E}^i_{xy}(s))\big]\max\limits_{t,s\in  \mathcal{I}_i}\bigg\{\frac{1}{2}\big[\bm{E}^i_{xy}(t) - \bm{E}^i_{xy}(s)\big]^2\bigg\} \;ds \; dt\bigg)}{\big[\int_{ \mathcal{I}_i} \exp(\bm{z}^i_{xy}\cdot \bm{E}_{xy}(t))dt\big]^2}\\
		&= \max\limits_{t,s\in  \mathcal{I}_i}\bigg\{\frac{1}{2}\big[\bm{E}^i_{xy}(t) - \bm{E}^i_{xy}(s)\big]^2\bigg\}\frac{\bigg(\iint\limits_{ \mathcal{I}_i \times  \mathcal{I}_i} \big[\exp(\bm{z}^i_{xy} \cdot \bm{E}^i_{xy}(t)) \exp(\bm{z}^i_{xy} \cdot \bm{E}^i_{xy}(s))\big] \;ds \; dt\bigg)}{\big[\int_{ \mathcal{I}_i} \exp(\bm{z}^i_{xy}\cdot \bm{E}^i_{xy}(t))dt\big]^2}\\
		&= \max\limits_{t,s\in  \mathcal{I}_i}\bigg\{\frac{1}{2}\big[\bm{E}^i_{xy}(t) - \bm{E}^i_{xy}(s)\big]^2\bigg\} . 
		\end{align*}
		Hence
		\[
		0 < (\bm{g}^i_{xy}(\bm{z}_{xy}))'' \leq \max\limits_{t,s\in  \mathcal{I}_i}\bigg\{\frac{1}{2}\big[\bm{E}^i_{xy}(t) - \bm{E}^i_{xy}(s)\big]^2\bigg\} .
		\]
		We now conclude that $(\bm{g}^i_{xy}(\bm{z}_{xy}))''$ is bounded and strictly positive.
	\end{proof}
	In order to show the local convexity of $\mathcal{J}$, we study the structure of
	the Hessian of the first term in the definition of $\mathcal{J}$. 
	\begin{lemma}\label{lem:diag}
		The matrix 
		$\big\langle \text{\bfseries{Hess}}(\bm K^{-1} \bm A^\top\bm{b}_{xy}(\bm{z}_{xy})) + \text{\bfseries{Hess}}(\bm{g}_{xy}(\bm{z}_{xy})),  
		\bm{u}_{xy}(\bm{z}_{xy}) + \bm{g}_{xy}(\bm{z}_{xy}) - \bm{d}_{xy} \big\rangle$ 
		is a diagonal matrix. 
	\end{lemma}
	\begin{proof}
		Without loss of generality, suppose that $n_t = 3$.
		Recall by our definition:
		\begin{align*}
		\bm{g}_{xy}(\bm{z}_{xy}) &= 
		\begin{pmatrix}
		\log(\frac{1}{| \mathcal{I}_1 |}\int_{ \mathcal{I}_1 }\exp(\bm{z}^1_{xy} \cdot \bm{E}^1_{xy}(t))dt)\\
		\log(\frac{1}{| \mathcal{I}_2|}\int_{ \mathcal{I}_2}\exp(\bm{z}^2_{xy} \cdot \bm{E}^2_{xy}(t))dt)\\
		\log(\frac{1}{| \mathcal{I}_3|}\int_{ \mathcal{I}_3}\exp(\bm{z}^3_{xy} \cdot \bm{E}^3_{xy}(t))dt)\\
		\end{pmatrix}.
		\end{align*}
		Then, recall the expression of $\bm g'_{xy}(\bm z_{xy})$ from \eqref{eq:c_exp}.
		Now to calculate the Hessian of $\bm{g}$, we need to take the derivative with respect to each variable once again. This will result in $\text{\bfseries{Hess}}(\bm{g}_{xy}(\bm{z}_{xy})) \in \mathbb{R}^{n_t \times n_t \times n_t}$ and will be of the form:
		\begin{align*}
		\text{\bfseries{Hess}}(\bm{g}_{xy}(\bm{z}_{xy})) = \begin{Bmatrix}
		\begin{pmatrix}
		(\bm{g}^1_{xy}(\bm{z}_{xy}))'' & 0 & 0\\
		0 & 0 & 0\\
		0 & 0 & 0
		\end{pmatrix}, \; \begin{pmatrix}
		0 & 0 & 0\\
		0 & (\bm{g}^2_{xy}(\bm{z}_{xy}))'' & 0\\
		0 & 0 & 0
		\end{pmatrix}, \;  \begin{pmatrix}
		0 & 0 & 0\\
		0 & 0 & 0\\
		0 & 0 & (\bm{g}_{xy}^3(\bm{z}_{xy}))''
		\end{pmatrix}
		\end{Bmatrix},
		\end{align*}  
		where $(\bm{g}^i(\bm{z}_{xy}))''$ is defined in Theorem \ref{thm:gconvex} .
		Next we consider $\text{\bfseries{Hess}}(\bm{K}^{-1} \bm{A}^\top\bm{b}_{xy}(\bm{z}_{xy}))$. 
		Observe:
		\begin{align*}
		\text{\bfseries{Hess}}(\bm{K}^{-1} \bm{A}^\top\bm{b}_{xy}(\bm{z}_{xy})) &= \bm{K}^{-1} \bm{A}^\top \cdot \text{\bfseries{Hess}}(\bm{b}_{xy}(\bm{z}_{xy}))
		\end{align*}
		where 
		\begin{align*}
		\text{\bfseries{Hess}}(\bm{b}_{xy}(\bm{z}_{xy})) = \begin{Bmatrix}
		\begin{pmatrix}(\bm{g}^1_{xy}(\bm{z}_{xy}))'' & 0 & 0\\
		0 & 0 & 0
		\end{pmatrix} , \begin{pmatrix}
		0 & -(\bm{g}^2_{xy}(\bm{z}_{xy}))'' & 0 \\
		0 &  (\bm{g}^2_{xy}(\bm{z}_{xy}))'' & 0 
		\end{pmatrix}, \begin{pmatrix}
		(0 & 0 & 0\\
		0 & 0 & \bm{g}^3_{xy}(\bm{z}_{xy}))''
		\end{pmatrix}\end{Bmatrix},
		\end{align*}
		\begin{align*}
			\text{\bfseries{Hess}}(\bm{b}_{xy}(\bm{z}_{xy})) \in \mathbb{R}^{{n_t -1} \times n_t \times n_t},
		\end{align*}	
		and 
		\begin{align*}
		\bm{K}^{-1}\bm A^\top := \begin{pmatrix}
		| & | \\
		\bm \kappa_1 & \bm \kappa_2\\
		| & |
		\end{pmatrix} \in \mathbb{R}^{n_t \times (n_t -1)} .
		\end{align*}
		Here $\bm\kappa_i$, $i=1...n_t$ represent the columns of $\bm{K}^{-1}\bm A^\top$. 
		Then we have the following:\begin{align*}
		&\text{\bfseries{Hess}}(\bm{K}^{-1} \bm{A}^\top\bm{b}_{xy}(\bm{z}_{xy}))\\
		&= \begin{Bmatrix}
		(\bm{g}^1_{xy}(\bm{z}_{xy}))'' \cdot\begin{pmatrix}
		| & 0 & 0\\
		\bm \kappa_1 & 0 & 0\\
		| & 0 & 0
		\end{pmatrix},
		(\bm{g}^2_{xy}(\bm{z}_{xy}))'' \cdot \begin{pmatrix}
		0 & | &  0\\
		0 & \bm \kappa_2 - \bm \kappa_1 &  0\\
		0 & | &  0
		\end{pmatrix},
		(\bm{g}^3_{xy}(\bm{z}_{xy}))'' \cdot\begin{pmatrix}
		0 & 0 &| \\
		0 & 0 & \bm \kappa_3\\
		0 & 0 & |
		\end{pmatrix}
		\end{Bmatrix},
		\end{align*}
		with $\text{\bfseries{Hess}}(\bm{K}^{-1} \bm{A}^\top\bm{b}_{xy}(\bm{z}_{xy}) \in \mathbb{R}^{n_t \times n_t \times n_t}$.
		Next we can add $\text{\bfseries{Hess}}(\bm{g}_{xy}(\bm{z}_{xy}))$ and
		\linebreak 
		$\text{\bfseries{Hess}}(\bm{K}^{-1} \bm{A}^\top\bm{b}_{xy}(\bm{z}_{xy}))$ to get: 
		\begin{align*}
		&\text{\bfseries{Hess}}(\bm{g}_{xy}(\bm{z}_{xy})) + \text{\bfseries{Hess}}(\bm{K}^{-1} \bm{A}^\top\bm{b}_{xy}(\bm{z}_{xy}))\\
		&=
		\begin{Bmatrix}
		(\bm{g}^1_{xy}(\bm{z}_{xy}))'' \cdot\begin{pmatrix}
		|  & 0 & 0 \\
		\bm\gamma_1 & 0 & 0 \\
		| & 0 & 0
		\end{pmatrix},
		(\bm{g}^2_{xy}(\bm{z}_{xy}))'' \cdot\begin{pmatrix}
		0 & | &  0 \\
		0 & \bm\gamma_2 & 0 \\
		0 & | &  0
		\end{pmatrix},
		(\bm{g}^3_{xy}(\bm{z}_{xy}))'' \cdot\begin{pmatrix}
		0 & 0 & |  \\
		0 & 0 & \bm\gamma_3 \\
		0 & 0 & |
		\end{pmatrix}
		\end{Bmatrix},
		\end{align*}
		where $ \bm\gamma_i \in \mathbb{R}^{n_t}, i=1, \dots, n_t$.
		Lastly let $\bm{u}_{xy}(\bm{z}_{xy}) + \bm{g}_{xy}(\bm{z}_{xy})-\bm{d}_{xy}
		= \bm\alpha_{xy} \in \mathbb{R}^{n_t}$.
		Then
		\begin{align*}
		&\big\langle\text{\bfseries{Hess}}(\bm{K}^{-1} \bm{A}^\top\bm{b}_{xy}(\bm{z}_{xy})) + \text{\bfseries{Hess}}(\bm{g}_{xy}(\bm{z}_{xy})),  
		\bm{u}_{xy}(\bm{z}_{xy}) + \bm{g}_{xy}(\bm{z}_{xy}) - \bm{d}_{xy} \big\rangle\\
		&= \begin{Bmatrix}
		(\bm{g}^1_{xy}(\bm{z}_{xy}))'' \cdot\begin{pmatrix}
		\bm\gamma_1^\top \bm\alpha_{xy} \\
		0\\
		0
		\end{pmatrix},
		(\bm{g}^2_{xy}(\bm{z}_{xy}))'' \cdot\begin{pmatrix}
		0\\
		\bm\gamma_2^\top \bm\alpha_{xy} \\
		0
		\end{pmatrix},
		(\bm{g}^3_{xy}(\bm{z}_{xy}))'' \cdot\begin{pmatrix}
		0\\
		0\\
		\bm\gamma_3^\top \bm\alpha_{xy} 
		\end{pmatrix}
		\end{Bmatrix}\\
		&\cong 	\begin{pmatrix}
		(\bm{g}^1_{xy}(\bm{z}_{xy}))'' \cdot\bm\gamma_1^\top \bm\alpha_{xy} & 0 & 0\\
		0 & (\bm{g}^2_{xy}(\bm{z}_{xy}))'' \cdot\bm \gamma_2^\top \bm\alpha_{xy} & 0 \\
		0 & 0 & (\bm{g}^3_{xy}(\bm{z}_{xy}))'' \cdot\bm \gamma_3^\top \bm\alpha_{xy}
		\end{pmatrix}.\tag{\theequation}\label{eq:mat}
		\end{align*}
	Here $\cong$ denotes the mode 1 unfolding of the $\big\langle\text{\bfseries{Hess}}(\bm{K}^{-1} \bm{A}^\top\bm{b}_{xy}(\bm{z}_{xy})) + \text{\bfseries{Hess}}(\bm{g}_{xy}(\bm{z}_{xy})),  
			\bm{u}_{xy}(\bm{z}_{xy}) + \bm{g}_{xy}(\bm{z}_{xy}) - \bm{d}_{xy} \big\rangle$ tensor. Further details of tensor matrix multiplication can be found in \cite{doi:10.1137/07070111X}.
		The proof is complete.
	\end{proof}
	Next, we establish that $\mathcal{J}$ is strictly locally convex on any finite interval when 
	$\lambda_1$ is chosen appropriately. 
	\begin{theorem}\label{thm:interval} 
		Let $\mathcal{I}_i$ denotes the $i$-$th$ exposure time interval with  
			$i = 1,\dots, n_t$. If there exists $t \in  \mathcal{I}_i$, for all $i$, with 
			${\bm E}^i_{xy}(t)$ being nonzero then for any closed ball $\Omega = 
			\overline{\mathcal{B}_\delta(0)}$ centered at 0 and radius $\delta$, 
			there exists $\lambda_1 > 0$ such that for $\bm{z}_{xy} \in \Omega$, 
			$\mathcal{J}$ is strictly convex. 
	\end{theorem}
	\begin{proof}
		First we calculate \textbf{Hess}($\mathcal{J}$) as: 
		\begin{align*}
		\text{\bfseries Hess}(\mathcal{J}) 
		&= \big\langle \; \bm{u}_{xy}'(\bm{z}_{xy}) + \bm{g}_{xy}'(\bm{z}_{xy})), \bm{u}_{xy}'(\bm{z}_{xy}) 
		+ \bm{g}_{xy}'(\bm{z}_{xy})) \; \big\rangle \\ 
		&\quad+ \big\langle \text{\bfseries{Hess}}(\bm{K}^{-1} \bm{A}^\top\bm{b}_{xy}(\bm{z}_{xy})) + \textbf{Hess}(\bm{g}_{xy}(\bm{z}_{xy})),  
		\bm{u}_{xy}(\bm{z}_{xy}) + \bm{g}_{xy}(\bm{z}_{xy}) - \bm{d}_{xy} \big\rangle + \lambda_1 \bm{I}.
		\end{align*}
		We notice that the term $\big\langle \; \bm{u}_{xy}'(\bm{z}_{xy}) + \bm{g}_{xy}'(\bm{z}_{xy})), \bm{u}_{xy}'(\bm{z}_{xy}) + \bm{g}_{xy}'(\bm{z}_{xy})) \; \big\rangle$ is of them form $\bm{M}^\top\bm{M}$, which implies that it is symmetric positive semi-definite. Next, from Lemma~\ref{lem:diag} we recall that the term $\big\langle \text{\bfseries{Hess}}(\bm{K}^{-1} \bm{A}^\top\bm{b}_{xy}(\bm{z}_{xy})) +\textbf{Hess}(\bm{g}_{xy}(\bm{z}_{xy})),  \bm{u}_{xy}(\bm{z}_{xy}) + \bm{g}_{xy}(\bm{z}_{xy}) - \bm{d}_{xy} \big\rangle$ is a diagonal matrix. In the case where the diagonal entries (eigenvalues) are greater than 0, then we are done. Suppose some eigenvalue is less than 0 and let 
		$ \bm{z}_{xy}\in \Omega := \overline{\mathcal{B}_\delta(0)} $ which implies $ \|\bm{z}_{xy}\|_\infty \leq \delta$. 
		Given $\delta > 0$, choose $\lambda_1 > 0$ such that
			\[
			0 < \delta 
			< \frac{\lambda_1 - (\|\bm\gamma\|_\infty \cdot M_1 \cdot \| \bm d_{xy}\|_\infty)(2\|\bm{K}^{-1}\bm{A}^\top\|_\infty + 1)}{(\|\bm\gamma\|_\infty \cdot M_1 \cdot M_2)(2\|\bm{K}^{-1}\bm{A}^\top\|_\infty ) + 1)} .
			\]
			Where $M_2 = \max\limits_i \bigg\{\big|\min\limits_{t \in  \mathcal{I}_i} \{\bm{E}^i_{xy}(t)\}\big|, \max\limits_{t \in  \mathcal{I}_i}\big\{\bm{E}^i_{xy}(t) \big\} \bigg\}$ and $M_1 = \max\limits_i \biggl\{\max\limits_{t,s\in  \mathcal{I}_i}\bigg\{\frac{1}{2}\big[\bm{E}^i_{xy}(t) - \bm{E}^i_{xy}(s)\big]^2\bigg\}\biggr\} .$
			Recall from Theorem \ref{thm:gconvex} that $0 < (\bm{g}^i_{xy}(\bm{z}_{xy}))'' \leq \max\limits_{t,s\in  \mathcal{I}_i}\bigg\{\frac{1}{2}\big[\bm{E}^i_{xy}(t) - \bm{E}^i_{xy}(s)\big]^2\bigg\}$, which implies $(\bm{g}^i_{xy}(\bm{z}_{xy}))'' \leq M_1$ for any $i$. 
			Then using \ref{eq:mat} we have the following: 
			\begin{align*}
			&\|\big\langle \text{\bfseries{Hess}}(\bm{K}^{-1} \bm{A}^\top\bm{b}_{xy}(\bm{z}_{xy})) + \textbf{Hess}(\bm{g}_{xy}(\bm{z}_{xy})),  \bm{u}_{xy}(\bm{z}_{xy}) + \bm{g}_{xy}(\bm{z}_{xy}) - \bm{d}_{xy} \big\rangle \|_\infty \\		
			&\leq
			\begin{Vmatrix}
			\begin{pmatrix}
			(\bm{g}^1_{xy}(\bm{z}_{xy}))'' \cdot\bm\gamma_1^\top \bm\alpha_{xy} & 0 & 0 &  0\\
			0 & (\bm{g}^2_{xy}(\bm{z}_{xy}))'' \cdot\bm \gamma_2^\top \bm\alpha_{xy} & 0 &  0 \\
			\vdots &  \vdots & \ddots & \vdots\\
			0 &  0 & \dots & (\bm{g}^{n_t}_{xy}(\bm{z}_{xy}))'' \cdot \bm{\gamma}_{n_t}^\top \bm\alpha_{xy}
			\end{pmatrix}\\
			\end{Vmatrix}_\infty
			 \\ 
			&\leq M_1 \cdot \|\bm\gamma\|_\infty \bigg(\|\bm{u}_{xy}(\bm{z}_{xy})\|_\infty + \|\bm{g}_{xy}(\bm{z}_{xy})\|_\infty + \|\bm{d}_{xy}\|_\infty\bigg)\\ 
			& \leq M_1 \cdot \|\bm\gamma\|_\infty \bigg(\|\bm{K}^{-1}\bm{A}^\top\bm{b}_{xy}\|_\infty + M_2 \|\bm{z}_{xy}\|_\infty + \|\bm{d}_{xy}\|_\infty \bigg)\\
			& \leq M_1 \cdot \|\bm\gamma\|_\infty \bigg(\|\bm{K}^{-1}\bm{A}^\top\|_\infty \|\bm{b}_{xy}\|_\infty + M_2 \|\bm{z}_{xy}\|_\infty + \|\bm{d}_{xy}\|_\infty \bigg)\\
			& \leq M_1 \cdot \|\bm\gamma\|_\infty \bigg(\|\bm{K}^{-1}\bm{A}^\top \|_{\infty} \; (2\|\bm d_{xy}\|_\infty +2M_2 \|\bm{z}_{xy}\|_\infty) + M_2 \|\bm{z}_{xy}\|_\infty + \|\bm{d}_{xy}\|_\infty \bigg)\\
			&< \lambda_1.
			\end{align*}
			Here 
			\[\bm{\gamma} := \begin{pmatrix}
			|  & |  & \dots&  |\\
			\bm\gamma_1 & \bm\gamma_2 & \dots& \bm\gamma_{n_t}\\
			| & |    & \dots &  |
			\end{pmatrix}^\top,\] $\bm{\gamma}_i$ and $\bm{\alpha}_{xy}$ are defined in Equation \ref{eq:mat} .
			Notice that $M_2$ is derived from bounds on $\bm{g}_{xy}(\bm{z}_{xy})$ as follows:
			\begin{align*}
			\bm{g}^i_{xy}(\bm{z}_{xy}) = \log\bigg(\frac{1}{| \mathcal{I}_i|}\int_{ \mathcal{I}_i}\exp (\bm{z}^i_{xy}\cdot \bm{E}^i_{xy}(t))dt\bigg) \leq \bm{z}^i_{xy}\cdot \max\big\{\bm{E}^i_{xy}(t)\big\}, \text{ where } \bm{z}^i_{xy}\cdot \bm{E}^i_{xy}(t)  \geq 0
			\end{align*}
			and
			\begin{align*}
			\bm{g}^i_{xy}(\bm{z}_{xy}) &= \log\bigg(\frac{1}{| \mathcal{I}_i|}\int_{ \mathcal{I}_i}\exp(\bm{z}^i_{xy}\cdot \bm{E}^i_{xy}(t))dt\bigg) \geq \bm{z}^i_{xy}\cdot \min\big\{\bm{E}^i_{xy}(t)\big\}, \text{ where } \bm{z}^i_{xy}\cdot \bm{E}^i_{xy}(t) \leq 0\\
			&\implies |\bm{g}^i_{xy}(\bm{z}_{xy})| \leq 		| \bm{z}^i_{xy}\cdot M_2 |.
			\end{align*}	
			Recall $\big\langle \text{\bfseries{Hess}}(\bm{K}^{-1} \bm{A}^\top\bm{b}_{xy}(\bm{z}_{xy})) + \textbf{Hess}(\bm{g}_{xy}(\bm{z}_{xy})),  \bm{u}_{xy}(\bm{z}_{xy}) + \bm{g}_{xy}(\bm{z}_{xy}) - \bm{d}_{xy} \big\rangle$ is diagonal from \ref{eq:mat} . Denote the diagonal entries as $\sigma_i$, $i=1, \dots, n_t$. Without loss of generality let $|\sigma_1| \geq |\sigma_i|$, $i=2, \dots, n_t$.
				Then we have the following: 	
				\[
				|\sigma_i| \leq |{\sigma_1}|= \|\big\langle \text{\bfseries{Hess}}(\bm{K}^{-1} \bm{A}^\top\bm{b}_{xy}(\bm{z}_{xy})) + \textbf{Hess}(\bm{g}_{xy}(\bm{z}_{xy})),  \bm{u}_{xy}(\bm{z}_{xy}) + \bm{g}_{xy}(\bm{z}_{xy}) - \bm{d}_{xy} \big\rangle \|_\infty < \lambda_1.
				\]
				Thus $\sigma_i < \lambda_1$ for all $i$. Furthermore $\sigma_i + \lambda_1 > 0$ for all $i$. 
				We notice \[
				\big\langle \text{\bfseries{Hess}}(\bm{K}^{-1} \bm{A}^\top\bm{b}_{xy}(\bm{z}_{xy})) + \textbf{Hess}(\bm{g}_{xy}(\bm{z}_{xy})),  \bm{u}_{xy}(\bm{z}_{xy}) + \bm{g}_{xy}(\bm{z}_{xy}) - \bm{d}_{xy} \big\rangle + \lambda_1 \bm{I}
				\] 
				is symmetric positive definite since it is a diagonal matrix with entries: $\sigma_i + \lambda_1 > 0$. Hence \textbf{Hess}($\mathcal{J}$) is symmetric positive definite due to it being the summation of a symmetric positive definite matrix and a symmetric positive semi-definite matrix.
				We conclude that for any convex set $\Omega$ there exists $\lambda_1$ such that $\mathcal{J}$ is strictly convex over $\Omega$.
	\end{proof}
	We further note that from Theorem \ref{thm:interval} once an interval is chosen a lower bound for $\lambda_1$ can be calculated by:
		\[
		\lambda_1 
		> \|\bm\gamma\|_\infty \cdot M_1 \bigg(\|\bm{K}^{-1}\bm{A}^\top \|_{\infty} \; (2\|\bm d_{xy}\|_\infty +2M_2 \|\bm{z}_{xy}\|_\infty) + M_2 \|\bm{z}_{xy}\|_\infty + \|\bm{d}_{xy}\|_\infty \bigg) .
		\]

\section{Experimental Setup}\label{s:numres}

This section focuses on how to prepare the dataset to be used in our computations.

\subsubsection{Standardizing $\bm{B^i}$}
	
	Depending on the event camera used, the entries of the matrix $\bm{B^i}$ from \eqref{eq:B}, where $i$ identifies a unique standard camera image in the sequence being considered, may take on various ranges to include $[0,255]$ or $[-255,255]$. We standardize $\bm{B^i}$ using MATLAB function mat2gray. This function will map the values of $\bm{B^i}$ to the range $[0,1]$ on a log scale. That being said since we define the matrix $\bm{d}^i = \log(\bm{B^i})$ we have to be careful not to allow $\bm{B^i}$ to have any zero values. We can easily protect against this by inserting a manual range into the mat2gray function of  $[\min{\bm{B^i}} -\epsilon, \max{\bm{B^i}}+\epsilon]$, for a sufficiently small $\epsilon$. In our computations we set $\epsilon = 1e-3$. We note $\min \bm B^i$ and $\max \bm{B}^i$ are the minimum and maximum values over all values in the matrix $\bm{B}^i$, respectively.

\subsubsection{Building the data cube} \label{s:cube}
In order to view the problem in three dimensions, we create a data cube as shown in Figure $\eqref{f:ex1}$. 
	As stated in \eqref{eq:cube} we define this cube by:
\begin{equation*}
{\bm{EC}}_{xy}^{s_j} = p_{x,y}^{s_j}.
\end{equation*}
Since the event data is generated on a nearly continuous basis, the data cube will be of size $\mathbb{R}^{\nu \times {n_x} \times {n_y}}$ where $\nu := \sum\limits_i n_i$ is the total number of events used in the reconstruction over all images. Note that $n_i$ is explained in further detail in \ref{eq:cube} . On average our image reconstructions use approximately $\nu \approx25,000$ events. It is impractical numerically to use this data cube due to its size. This issue is resolved by reducing the size of the data cube via time compression. Choose some $r,k \in \mathbb{N}$ such that $\nu \approx r\cdot k  $ and define a new data cube as follows:
	\begin{align*}
		\bm{dc}^{\pi} = \sum_{\omega = 1 + (\pi-1)\cdot k}^{\pi \cdot k} \bm{EC}^{\omega}
	\end{align*}
 Then we have $\bm{dc} \in \mathbb{R}^{r \times {n_x} \times {n_y}}$ where $r \approx \frac{\nu}{k}$, and $\pi = 1, \dots, r$. For our examples we have chosen $k \approx 200$.

\subsubsection{Calculate $\bm{g}(\bm{z})$} 
Once we have built the data cube then we can calculate the function $\bm{g}(\bm{z})$ which is defined componentwise as:  
\[
\bm{g}_{xy}^i(\bm{z}_{xy}) = \log \bigg( \frac{1}{| \mathcal{I}_i | }\int_{ \mathcal{I}_i} \exp \big(\bm{z}^i_{xy} \cdot \bm{E}^i_{xy}(\tau, t_i)\big)dt \bigg), \quad i=1, \dots, n_t.
\]

Numerically we approximate this using the trapezoid rule as: 
\[
\bm{g}_{xy}^i(\bm{z}_{xy}) \approx \log \bigg( \frac{1}{| \mathcal{I}_i | } \sum_{l=2}^{r}\big(\frac{\triangle \tau_{l-1}}{2}\big)\big(\exp \big(\bm{z}^i_{xy} \cdot \bm{E}^i(\tau_{l-1}, t_i)\big) + \exp \big(\bm{z}^i_{xy} \cdot \bm{E}^it_{xy}(\tau_{l}, t_i)\big)\big) \bigg) ,
\]
where $t_i$ is defined in \eqref{eq:Exy} and $\Delta \tau_{l-1} = \tau_l - \tau_{l-1}$ is the step length.  We evaluate the function $\bm{b}(\bm{z})$ as well as the derivatives of $\bm{b}(\bm{z})$ and $\bm{g}(\bm{z})$ in a similar way. 

\subsubsection{Omitting unnecessary calculations for $\bm{z}$} 
We recall that $\bm{z}$ must be calculated on a per pixel basis. To decrease computation time we only perform an optimization for $\bm{z}_{xy}$ over the pixels for which event data has been captured. 

\subsubsection{Image Reconstruction} \label{s:IR}
As shown in Figure~\ref{f:ex1} our method is capable of generating two image representations. Given some value ${\bm{z}}$ the $\bm{u}$ representation is given as: $\bm{u}({\bm{z}})$ defined in $\eqref{eq:u}$. The $\bm{u}$ representation reconstructs only the dynamical part of the image. This is due to our definition of $\bm{b}$ in \eqref{eq:b}. In order to reconstruct the full image including the portion of the image with no dynamics, we use \[\bm{v}_{xy}^i = \bm{d}_{xy}^i - \bm{g}_{xy}^i({\bm{z}}),\] a variation of the definition given in \eqref{eq:EDI}. In general we will refer to the latter formulation as the reconstruction, and will specifically denote the $\bm{u}$ representation when presented. 

\section{Numerical Examples}
\label{s:nex}

Next, we present a series of examples which establishes that the proposed approach
could be beneficial in practice.

\subsection{Benchmark Problem: Unit Bump}

The first example is a synthetic case where
we consider a blurred unit bump, see Figure~\ref{f:ex4} . In order to perform this analysis, we require a baseline image for comparison, a series of consecutive images as input for the ESIM generator \cite{Gehrig_2018} and our proposed method, a blurry image, and the corresponding event data.  To begin, we will construct a series of consecutive images by first creating a white circle on a black background (not shown). Using this initial image, we move the white circle two pixels to the right and two pixels down eight times. This set of nine images completes our series of images. We now generate a tenth image, the blurry image (middle in Figure~\ref{f:ex4}), by averaging the sequence of nine images. The fifth image of the set is designated as our baseline image (left in Figure~\ref{f:ex4}). Lastly, to generate the associated event data, we use the ESIM generator from \cite{Gehrig_2018}. The ESIM generator takes a series of images as input, and will output event data. Recall that our model requires a series of images as well as event data in order to generate a reconstructed image. We set $\bm{B}^1$ to be the fourth image of the nine image set, $\bm{B}^2$ to be the blurry image, and $\bm{B}^3$ to be the sixth image. The result of our method (right in Figure~\ref{f:ex4}) is compared to the baseline image. As shown, we obtain a high quality reconstruction with SSIM of 0.96 and PSNR of 27.3. This example serves 
a benchmark for us to apply our approach on the event based dataset.
\begin{figure}[!htb]
	\centering
	\includegraphics[width=.2\textwidth]{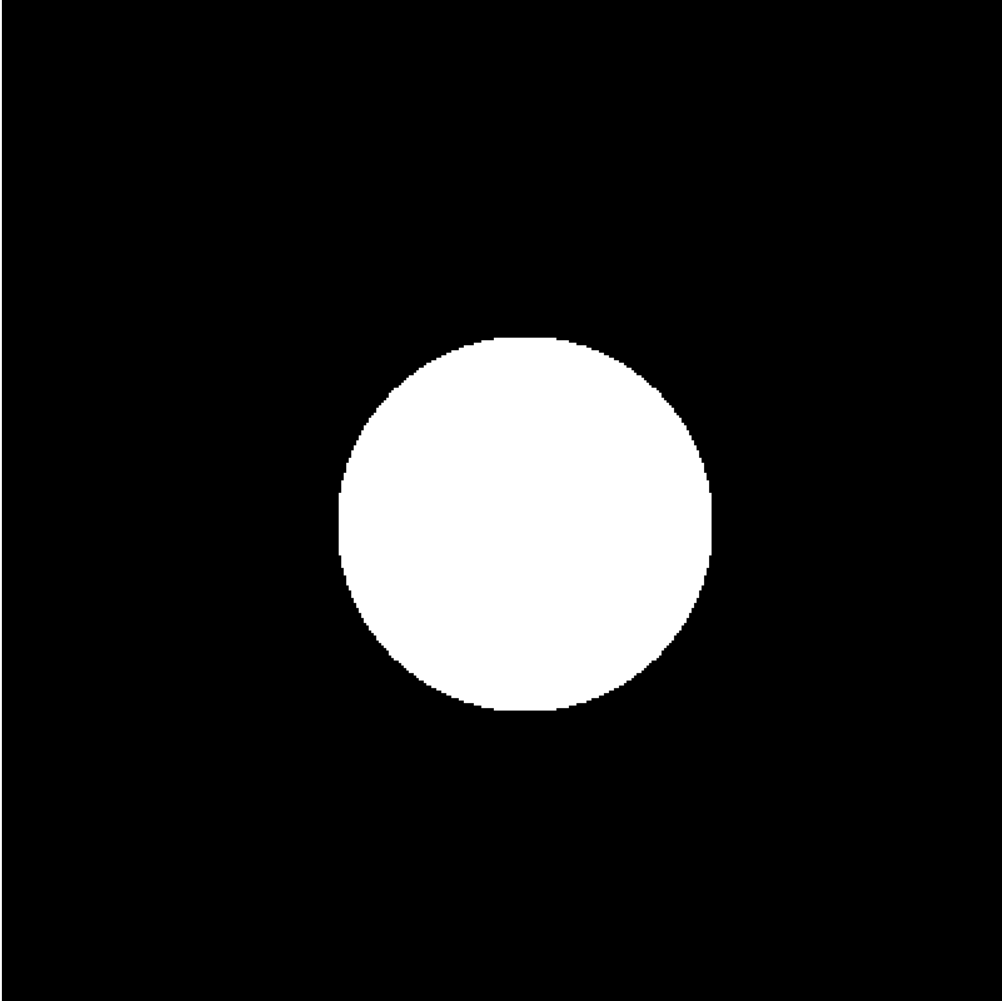} \qquad
	\includegraphics[width=.2\textwidth]{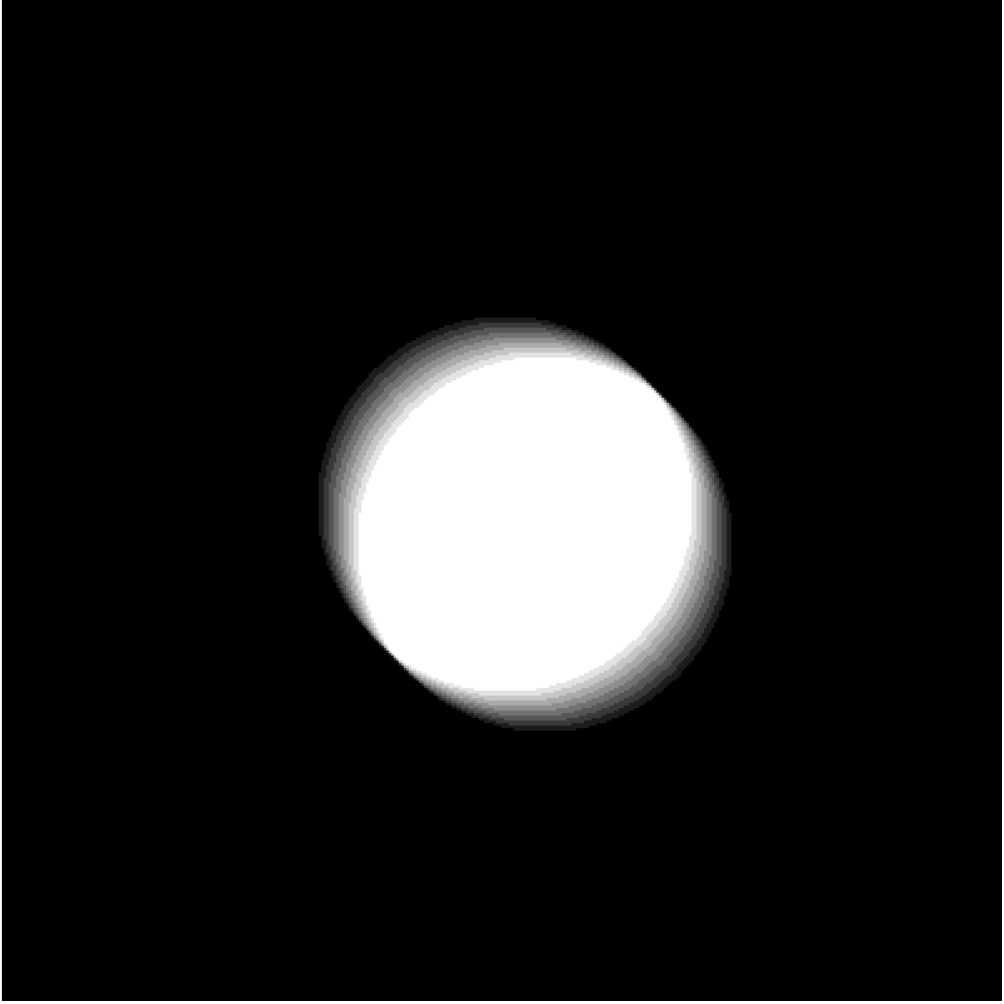} \qquad
	\includegraphics[width=.2\textwidth]{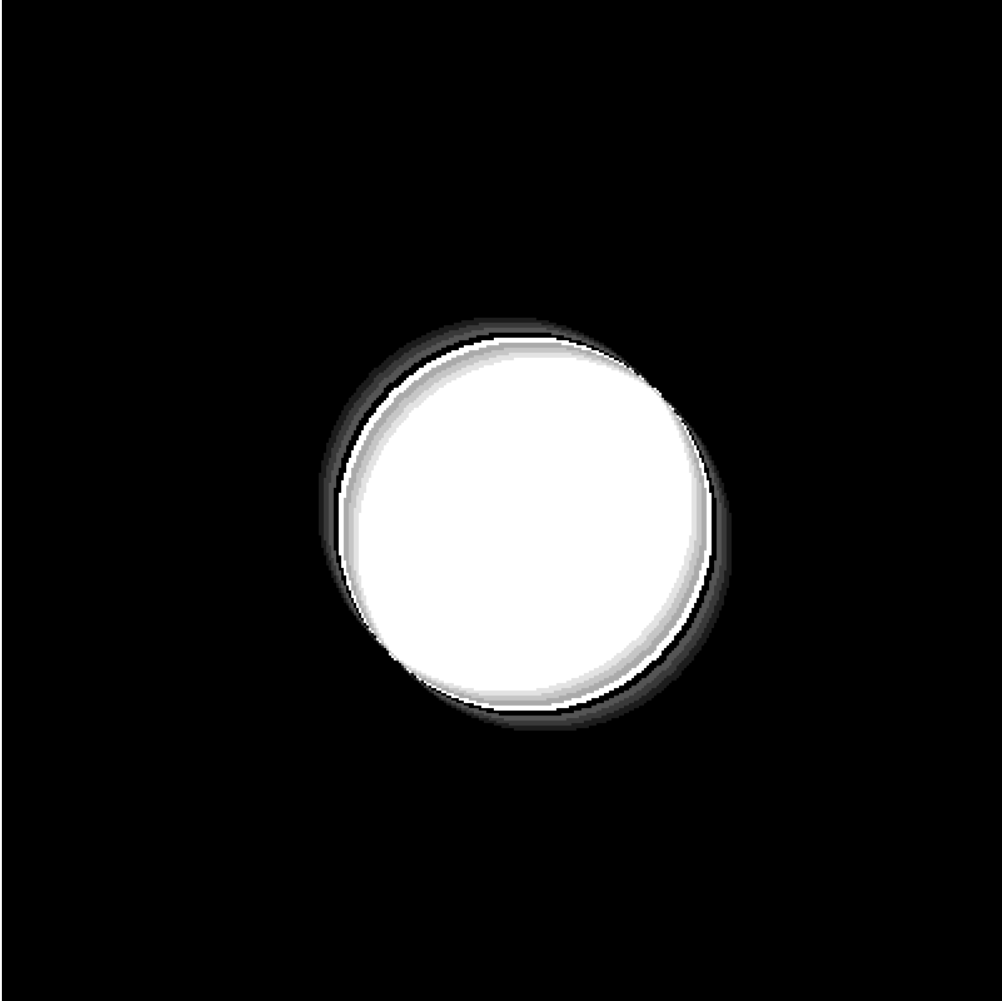}
	\caption{From left to right we have our baseline image, a simulated blurry image, and finally our reconstruction. Events for this reconstruction were produced using the ESIM generator from \cite{Gehrig_2018}. Using parameters $\lambda_1 = 1$ and $\lambda_2 = 1e-3$, we achieve an SSIM value of 0.96 and a PSNR value of 27.3.  }
	\label{f:ex4}
\end{figure}

\subsection{Multiple Event Based Dataset Examples}
Figures~\ref{f:ex1}-\ref{f:ex3} show the application of our method to various examples. In each of these examples three standard camera images are used to define the $\bm{B}$ variable, while the $\bm{g}$ and $\bm{b}$ variables are calculated using the associated event data. In each of these examples the $\bm{B}^2$ image was chosen to be a blurry image. We will refer to the $\bm{B}^2$ image as the blurry image for the remainder of this section.

In Figure \ref{f:ex1} for the chosen three image sequence $\approx 39,000$ events were recorded. The blurry image is shown (left) while our reconstruction is shown (right). Our reconstruction deblurs the fence-line as well as captures the grass and tree texture that is not apparent in the blurred image.
The blurred image (left) in Figure \ref{f:ex2} shows a person swinging a pillow. Approximately 26,000 events were captured for this example. Our model has successfully reconstructed the boundary of the pillow as well as the persons hand and thumb.
Figure \ref{f:ex3} is an example of a person jumping and the three image sequence used has $\approx$ 16000 associated events. Notice in the blurred image (left) the persons arm is barely visible. In contrast our model is able to reconstruct the persons arm (right). 
\begin{figure}[!htb]
	\centering
	\includegraphics[width=.24\textwidth]{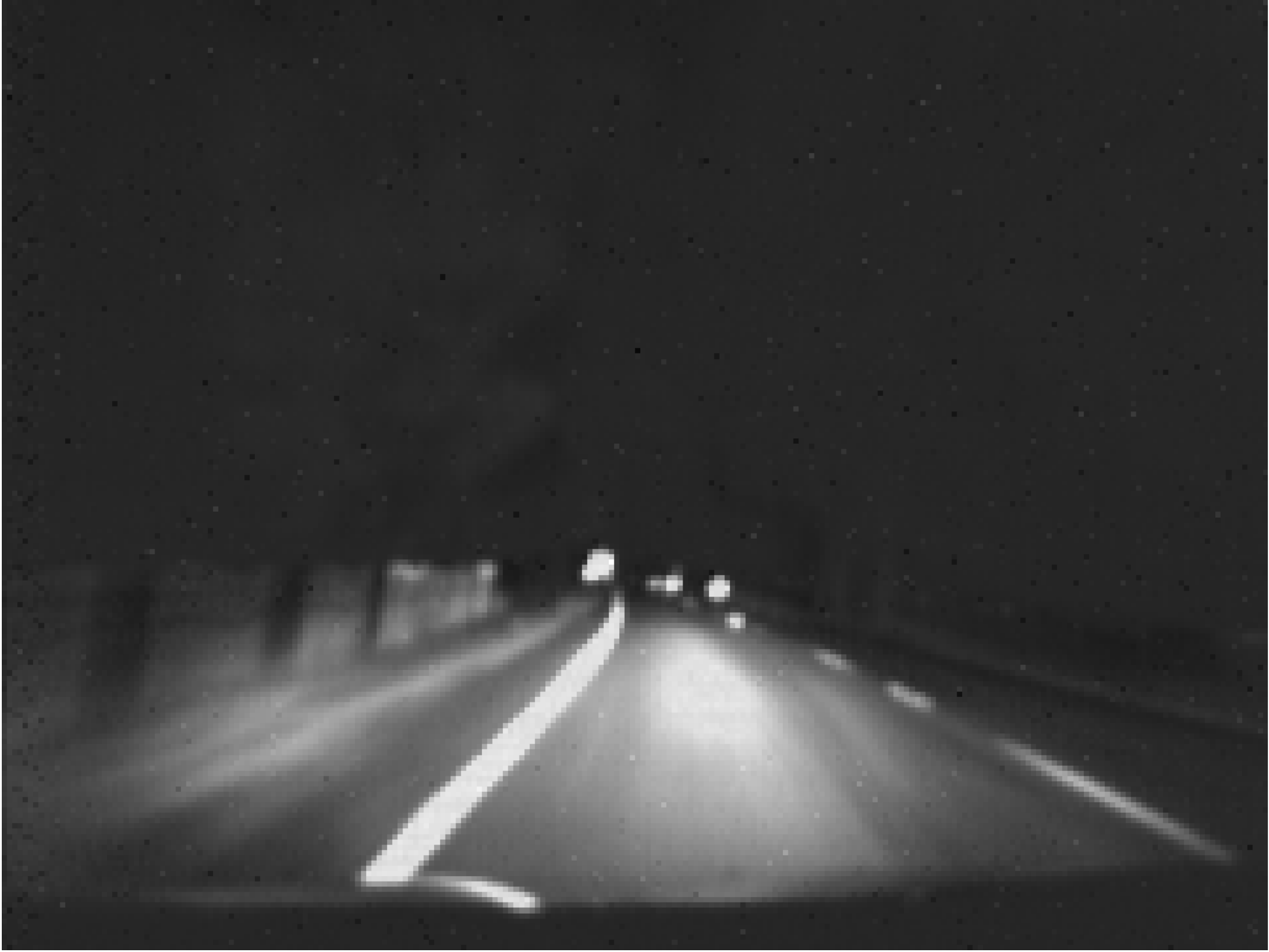}
	\includegraphics[width=.24\textwidth]{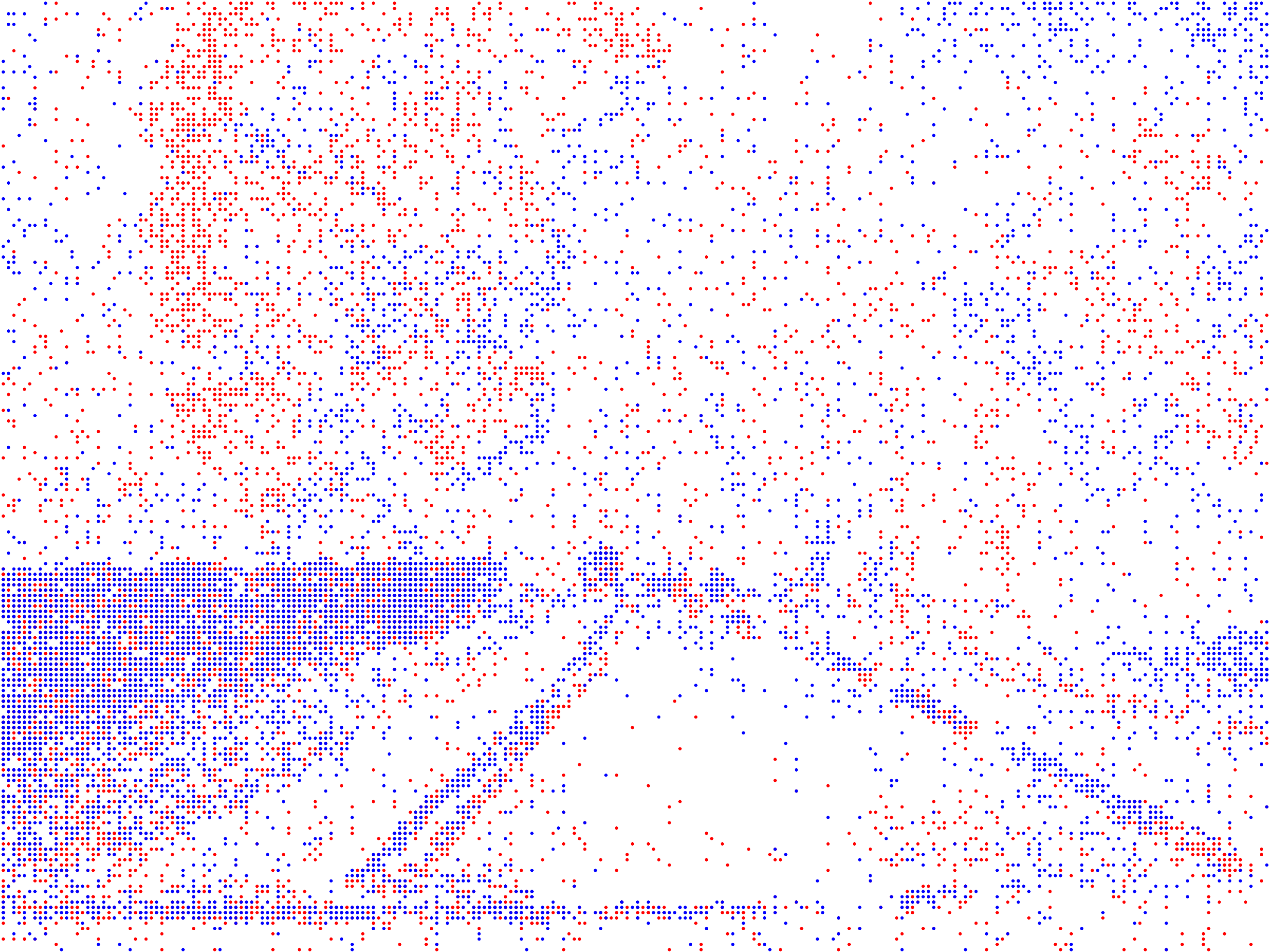}
	\includegraphics[width=.24\textwidth]{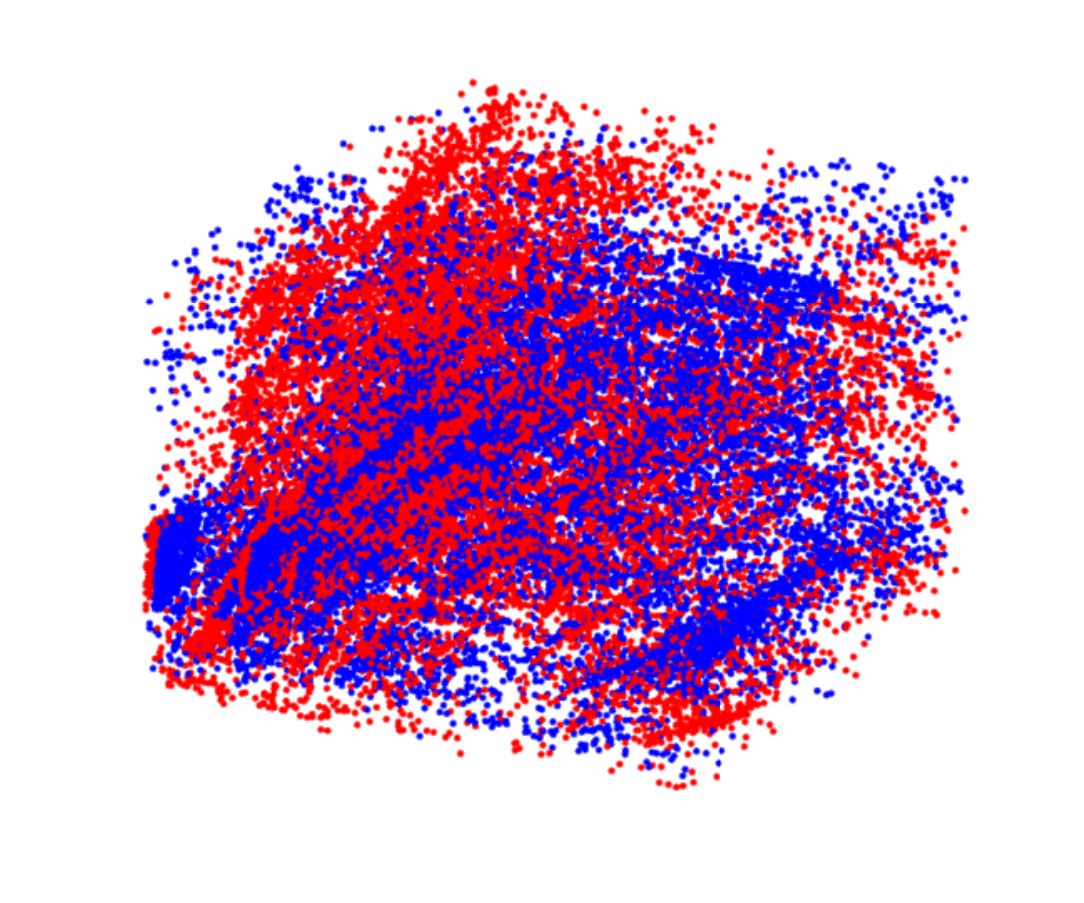}
	\includegraphics[width=.24\textwidth]{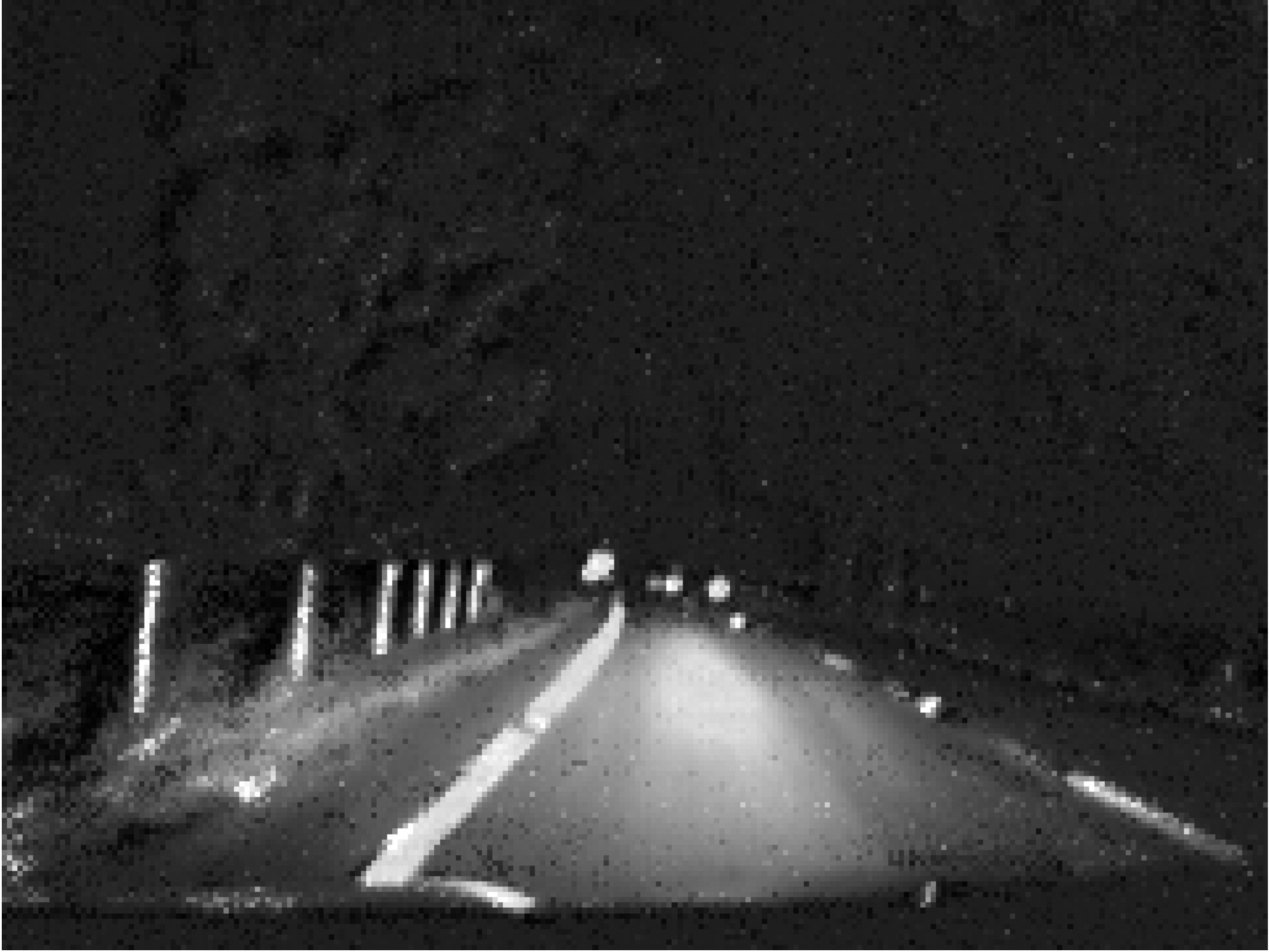}
	\caption{Example of our method with a night drive data set. From left to right we have a standard camera image that contains motion blur, a 2D representation of the event data color coded red for positive events and blue for negative events, a 3D representation of the events over time, and finally the reconstructed image using our bilevel optimization method. Here, we have used $\lambda_1 =1$ and $\lambda_2 = 1e-3$.  }
	\label{f:ex1}
\end{figure}

\begin{figure}[!htb]
	\centering
	\includegraphics[width=.32\textwidth]{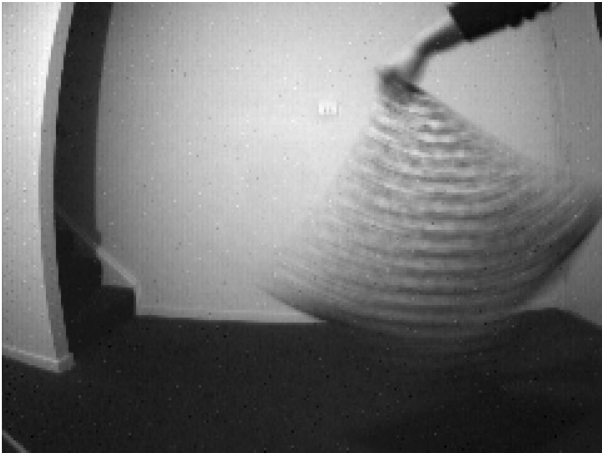}
	\includegraphics[width=.32\textwidth]{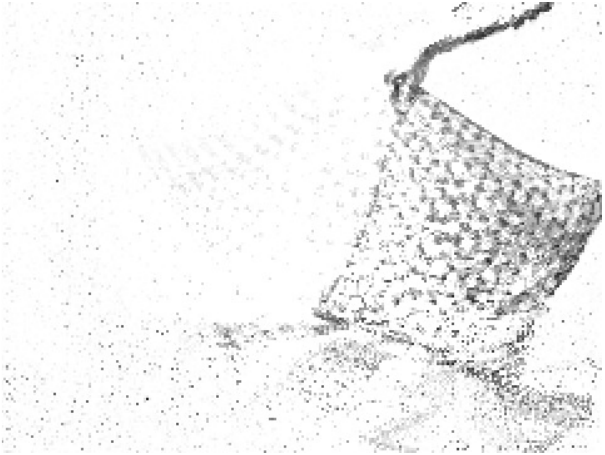}
	\includegraphics[width=.32\textwidth]{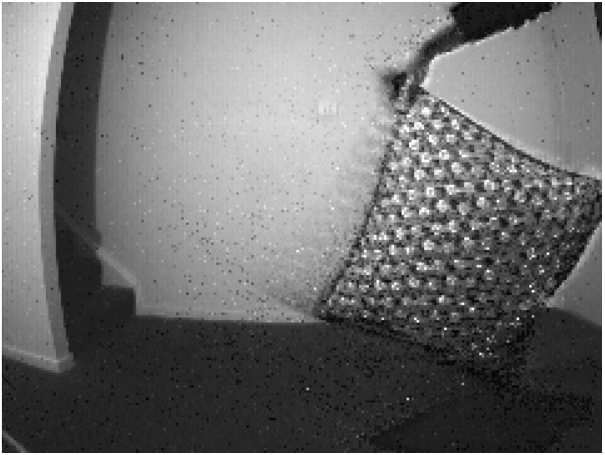}
	\caption{From left to right we have the original blurred image from a standard camera, the $\bm{u}$ representation, and the reconstruction using our bilevel method. This reconstruction was generated with $\lambda_1 = 0.5$ and $\lambda_2 = 1e-3$. }
	\label{f:ex2}
\end{figure}
\begin{figure}[!htb]
	\centering
	\includegraphics[width=.31\textwidth]{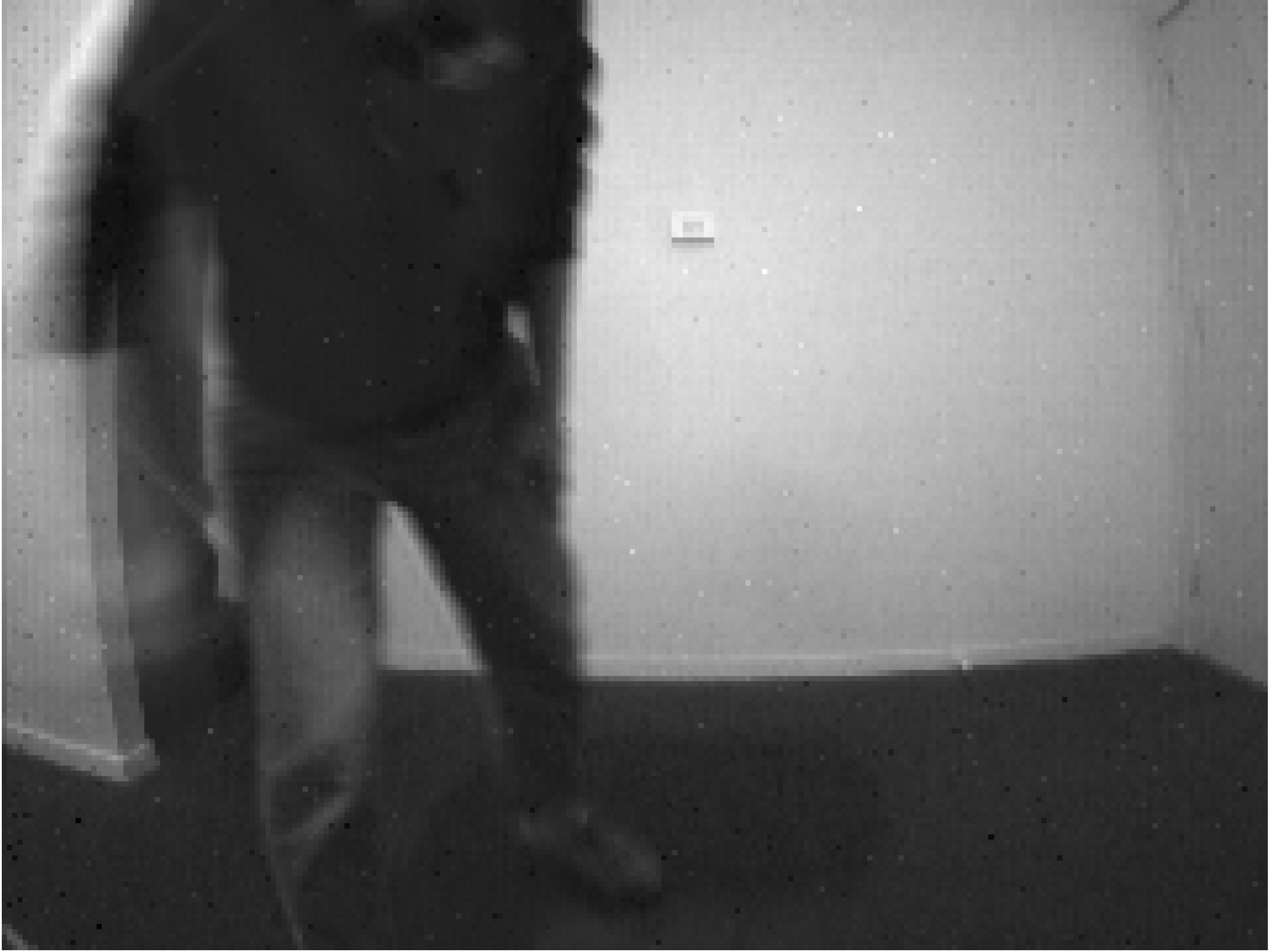}
	\includegraphics[width=.31\textwidth]{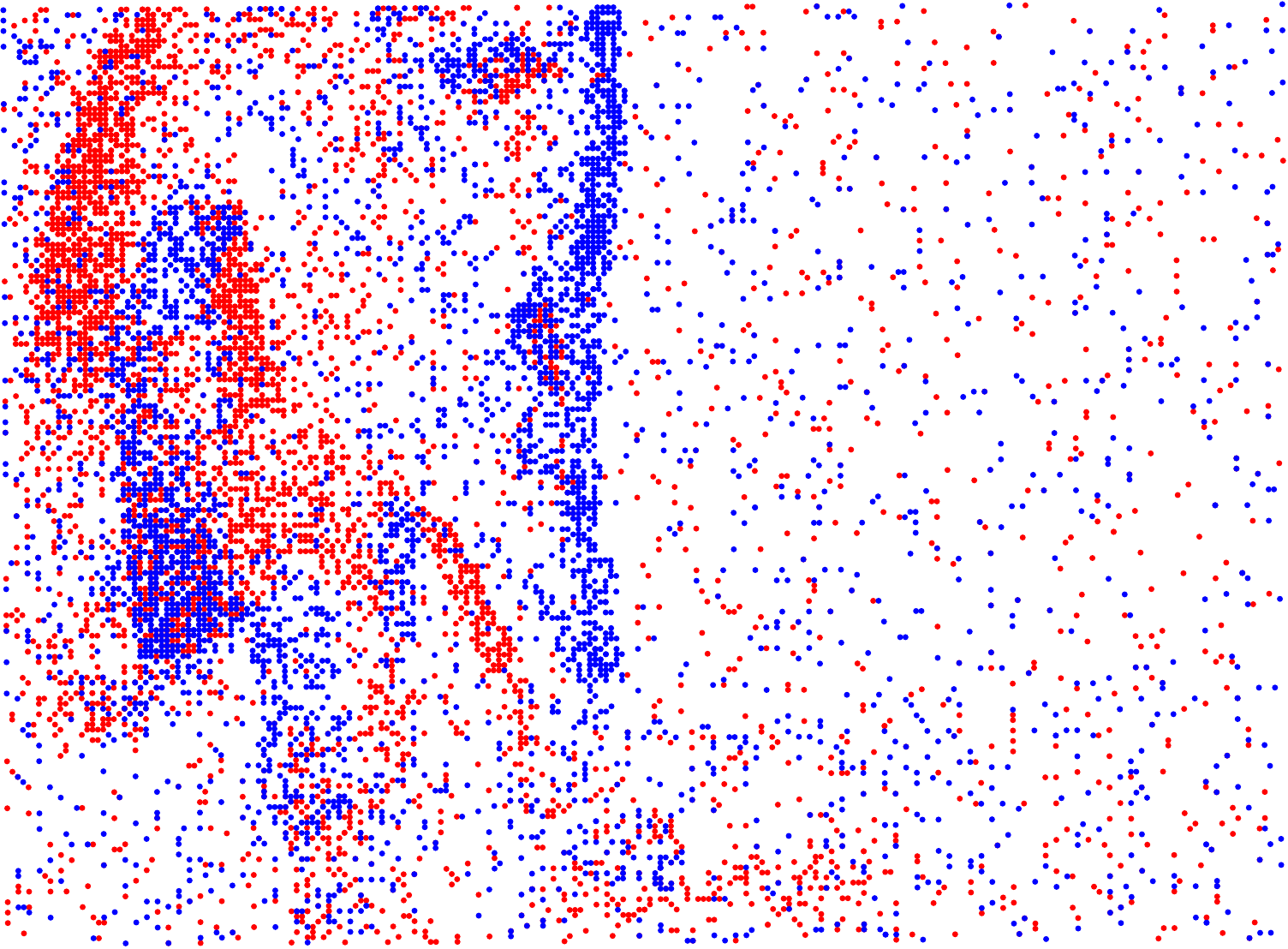}
	\includegraphics[width=.31\textwidth]{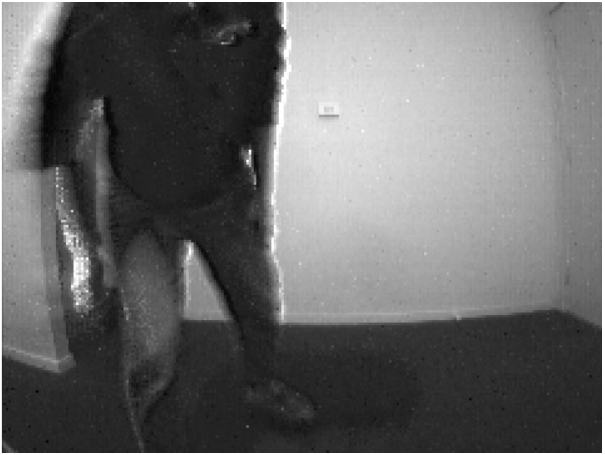}
	\caption{From left to right we have a standard blurry camera image, the recorded events, and finally our reconstructed image. This reconstruction was generated with $\lambda_1 = 1.5$ and $\lambda_2 = 1e-3$.}
	\label{f:ex3}
\end{figure}

\subsection{Event Based Problem: Shadowing Effects}\label{s:shadow}
As noted in Section \ref{s:existing} one of the limitations to the mEDI model is the appearance of shadows in a high contrasting environment. Examples of this shadowing effect can be seen in \cite{pan2019bringing} and Figures \ref{f:nrun} and \ref{f:nrun2}. As shown in Figures \ref{f:nrun} and \ref{f:nrun2} our model is able to substantially reduce the shadowing effects of the EDI and mEDI models.\begin{figure}[!htb]
	\centering
	\includegraphics[width=.25\textwidth]{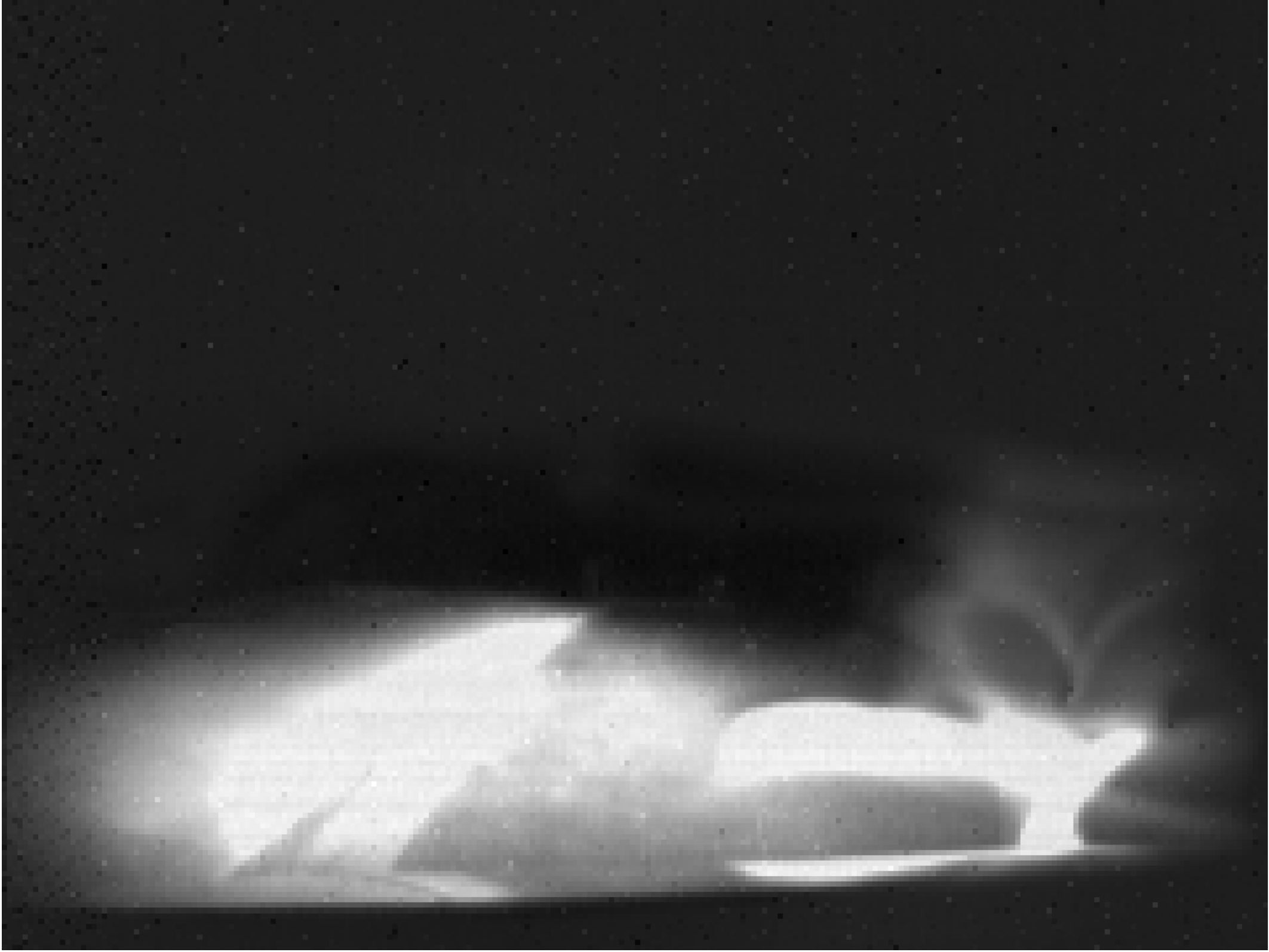}  \qquad
	\includegraphics[width=.25\textwidth]{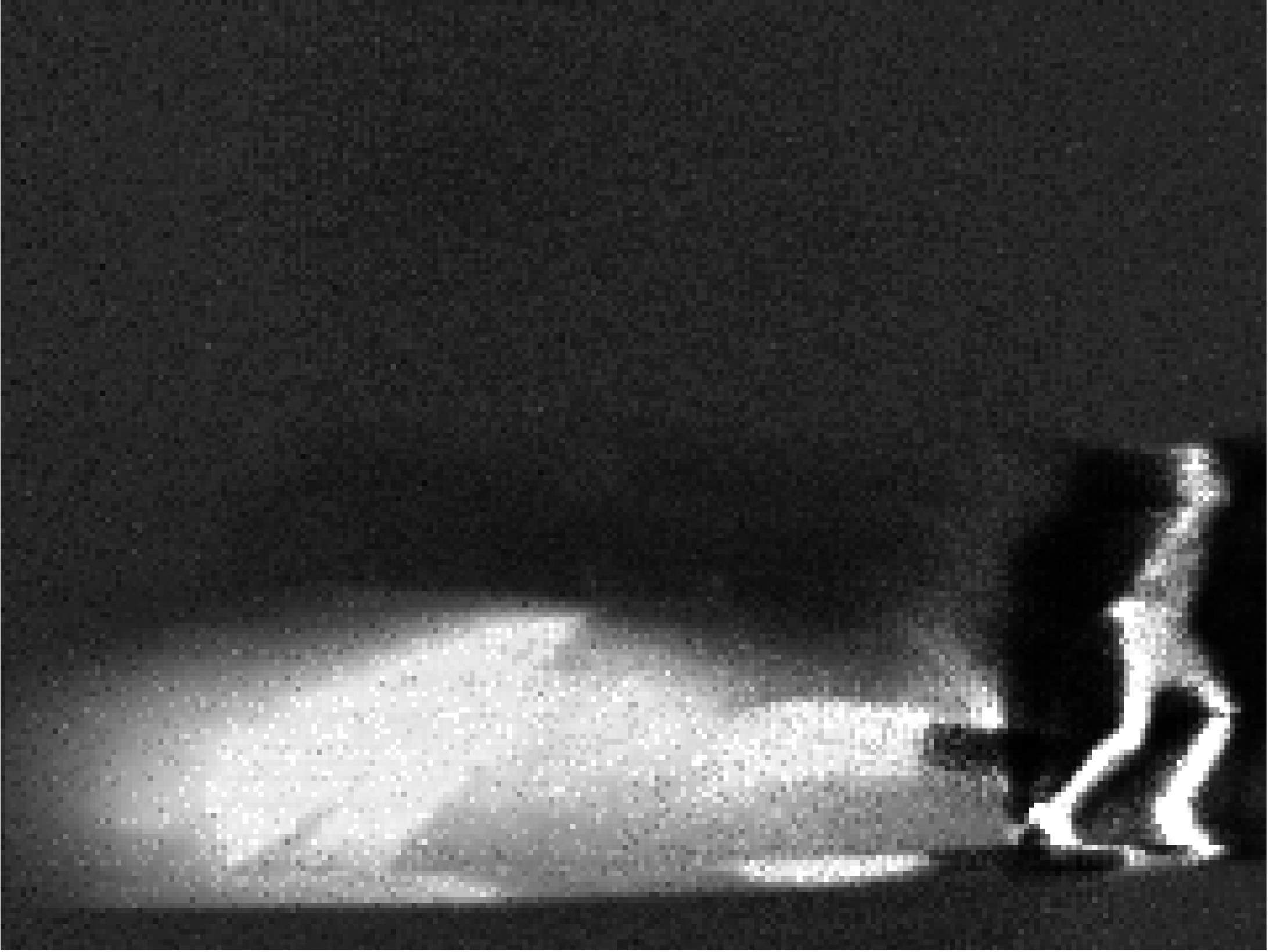} \qquad
	\includegraphics[width=.25\textwidth]{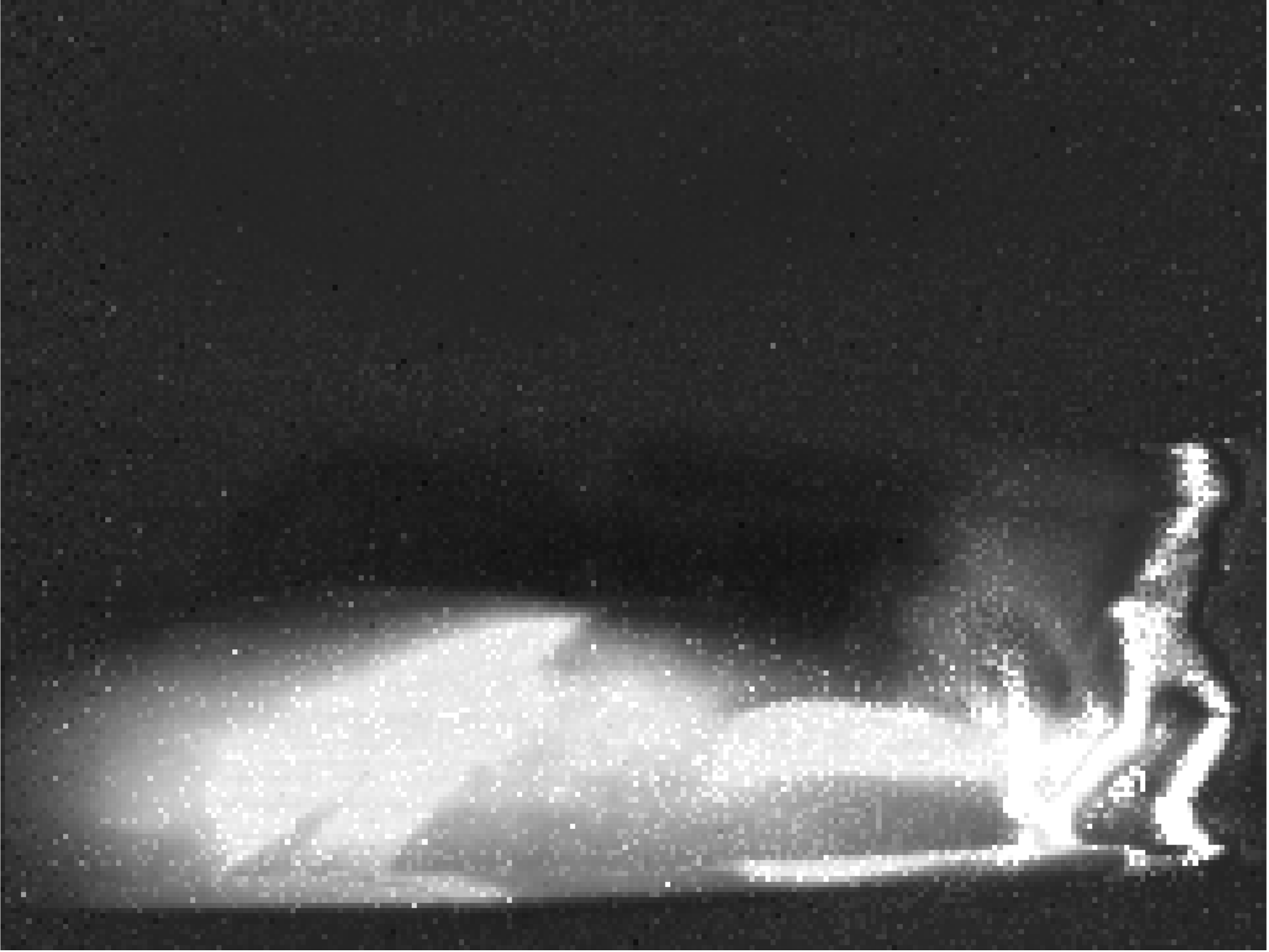}
	\caption{In this figure we see the shadowing effects that can occur due a high contrast environment. From left to right we have the standard camera image, the image produced by the mEDI method, and finally the image produced via our bilevel optimization method with $\lambda_1 = 1$ and $\lambda_2 = 1e-3$. Clearly, our method helps overcome the shadowing effect.}
	\label{f:nrun}
\end{figure}
\begin{figure}[!htb]
	\centering
	\includegraphics[width=.25\textwidth]{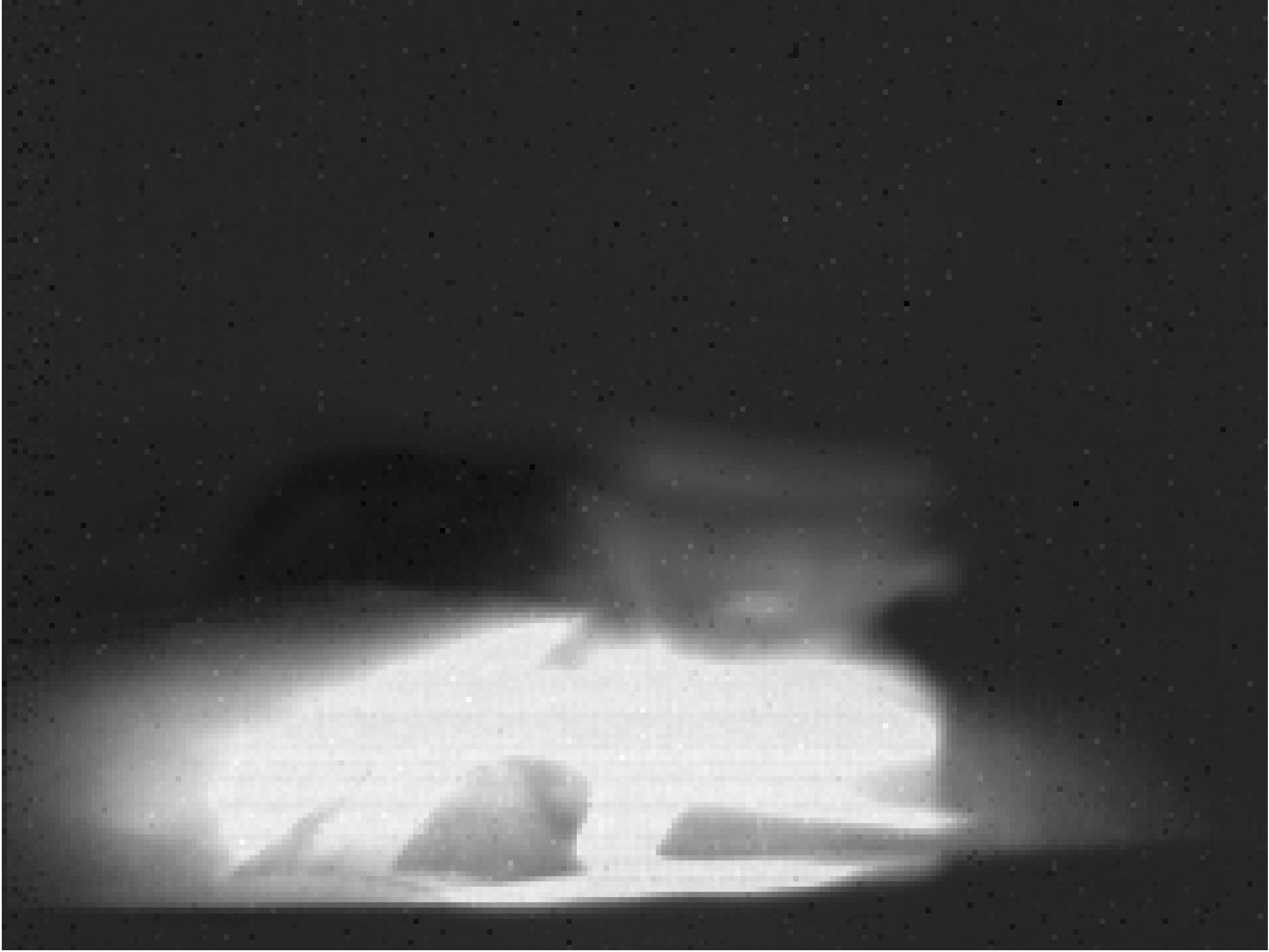} \qquad
	\includegraphics[width=.25\textwidth]{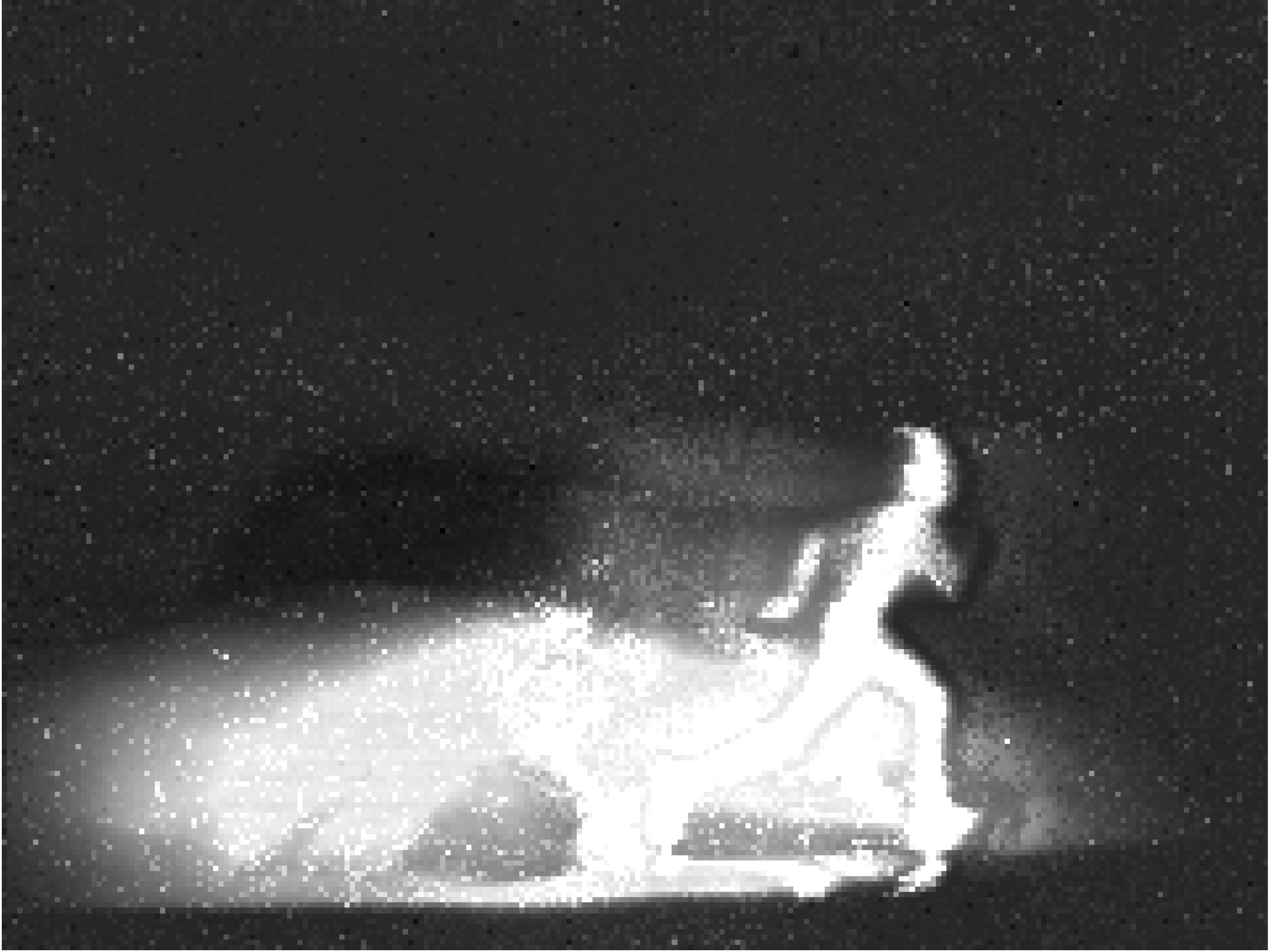} 
	\caption{The panels show a second example of our method (right) used to deblur the a standard camera image(left) resulting in a reduced amount of shadowing compared to the existing approaches.}
	\label{f:nrun2}
\end{figure}Figure~\ref{f:plot1} shows the behavior of five randomly chosen pixels. 
We observe the typical convergence behavior of a Newton's method.
\begin{figure}[!htb]
	\centering
	\includegraphics[width=.5\textwidth]{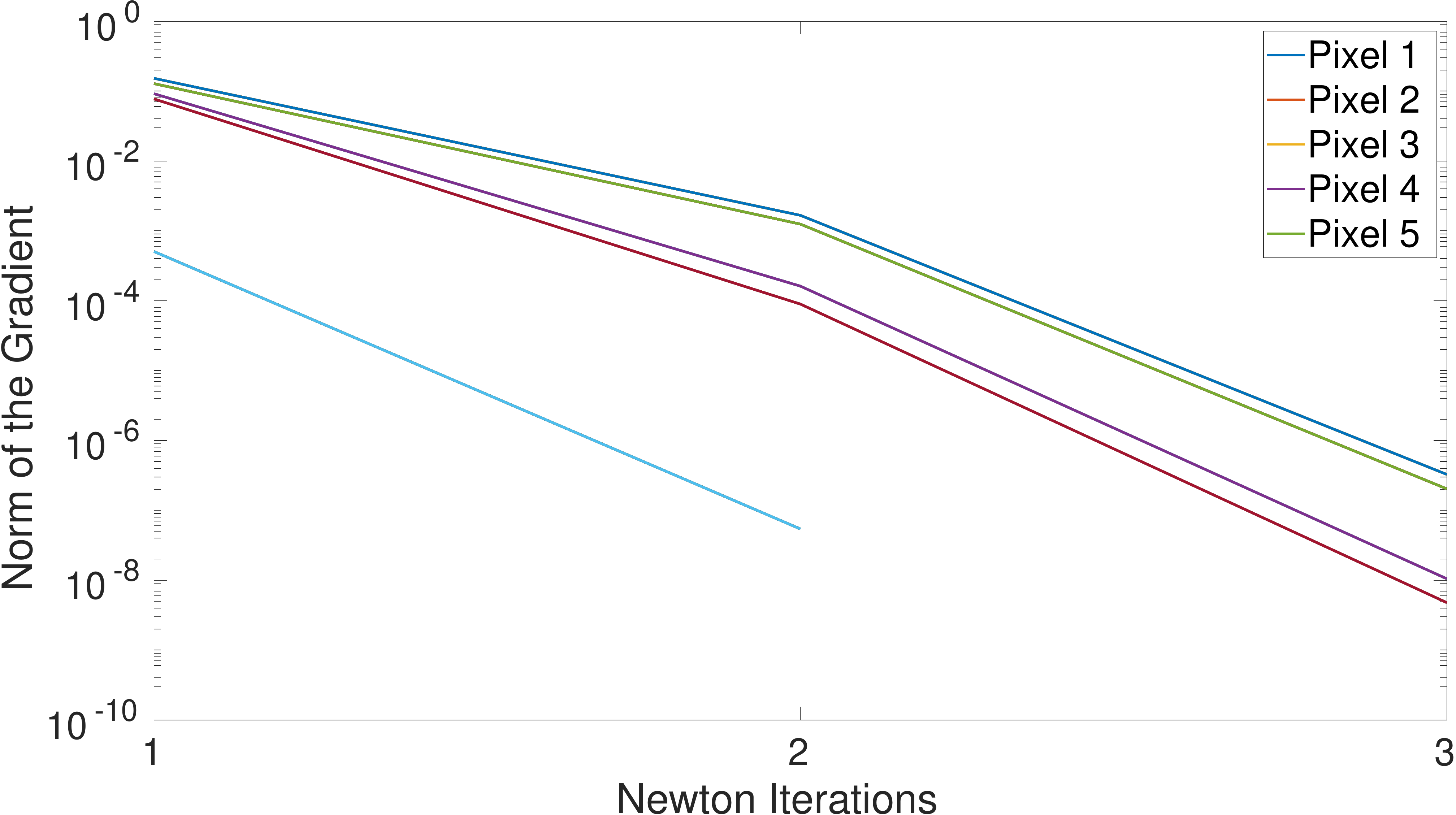}
	\caption{This figure shows the value $\|\nabla \mathcal{J}(\bm{z})\|$ for each Newton iteration over a sample of pixels. The Newton iterations shown are from the optimization of pixels sampled from the image reconstruction that is shown in Figure \ref{f:nrun}.}
	\label{f:plot1}
\end{figure}


\section*{Conclusion, Limitation, and Future Work}

In this study, we introduce a novel approach to image deblurring by combining event data and multiple standard camera images within a variational framework. Our method demonstrates the existence of solutions for a locally convex problem and we apply a second order Newton solver to several examples, showcasing the effectiveness of our approach.

Though our method does offer a systematic way to construct de-blurred images across time there are limitations that still exist in the model. 
\begin{enumerate}
	\item Our method requires a pixel by pixel optimization and depending on the camera resolution can run into scalability issues. This issue would be most evident with images where dynamics occur at a majority of pixels.
	\item Manual tuning may be required of certain parameters as described in Section \ref{s:cube} .
\end{enumerate}
To mitigate the scalability risk, we propose to filter out pixels that have significantly lower event counts compared to other pixels. In future work, we also plan to test the model's robustness in a domain shift scenario, such as a camera moving in and out of water.  Moreover, we aim to develop a parameter learning framework to optimize the regularization parameters for better performance.

\section*{Acknowledgement}

The authors are grateful to Dr. Patrick O'Neil and Dr. Diego Torrejon (BlackSky) 
and  Dr. Noor Qadri (US Naval Research Lab, Washington DC), for bringing the 
neuromorphic imaging topic to their attention and for several fruitful discussions.

\bibliographystyle{plain}
\bibliography{refs.bib}

\end{document}